%% file: BaDeSc25-pre.tex
\documentclass{amsart}

\usepackage[hmarginratio=1:1]{geometry}

\usepackage{amsmath}
\usepackage{amssymb}
\usepackage{amsthm}
\usepackage{esint}

\usepackage{datatool}
\usepackage{graphicx}
\usepackage{capt-of}
\usepackage{multirow}
\usepackage{booktabs}

\usepackage{color}

\usepackage{hyperref}

\usepackage[backend=biber, style=alphabetic, sorting=nyt, url=false, isbn=false, giveninits=true]{biblatex}

\include{macros_ds}

\allowdisplaybreaks

\title[Quasi-optimal error estimate for inextensible elastic curves]{Quasi-optimal error estimate for the approximation of the elastic flow  of inextensible curves}

\author{Sören Bartels}
\address{Abteilung für angewandte Mathematik, Albert-Ludwigs-Universität Freiburg, Hermann-Herder-Str.~10, 79104 Freiburg i.~Br., Germany}
\email{bartels@mathematik.uni-freiburg.de}

\author{Klaus Deckelnick}
\address{Institut für Analysis und Numerik, Otto-von-Guericke-Universität Magdeburg, Universitätsplatz~2, 39106 Magdeburg, Germany}
\email{klaus.deckelnick@ovgu.de}

\author{Dominik Schneider}
\address{Abteilung für angewandte Mathematik, Albert-Ludwigs-Universität Freiburg, Hermann-Herder-Str.~10, 79104 Freiburg i.~Br., Germany}
\email{dominik.schneider@mathematik.uni-freiburg.de}

\date{\today}


\begin{document}
  \begin{abstract}
    A space-discretization for the elastic flow of inextensible curves is devised and quasi-optimal convergence of the corresponding semi-discrete problem is proved for a suitable discretization of the nonlinear inextensibility constraint. Further a fully discrete time-stepping scheme that incorporates this constraint is proposed and unconditional stability and convergence of the discrete scheme are proved. Finally some numerical simulations are used to verify the obtained results experimentally.
  \end{abstract}
  
  \maketitle  
  % introduction
  \section{Introduction}
  \label{section:Introduction}
  Given an interval $I = (a,b)$ and an arc-length parameterized curve $u: I \to \real^d$, in absence of twist, its bending energy $E(u)$ is given by
  \begin{equation*}
    E(u) = \frac{1}{2} \int_I \vert u'' \vert^2 \ dx.
  \end{equation*}
  This model goes back to Bernoulli and can be derived as a special case by dimension reduction from three-dimensional hyperelasticity, see~\cite{MM03,Bar20}.
  We are interested in finding energy-decreasing evolutions for this energy functional under given Dirichlet boundary conditions $u = u_D$ on $\Gamma_D \subset \{a\}$, $u' = u_D'$ on $\Gamma_D' \subset \del I$ and the arc-length constraint $\vert u' \vert^2 = 1$ in $I$.
  The first variation of the energy functional yields the Euler--Lagrange equation
  \begin{equation*}
    0 = \int_I u'' \cdot v'' \ dx
  \end{equation*}
  for all tangential fields $v$ satisfying homogeneous boundary conditions and the linearized arc-length constraint $u' \cdot v' = 0$.
  The elastic flow is then defined as the $L^2$ gradient flow of $E$.
  Thus if $z \in H^1([0,T];L^2(I)^d) \cap L^\infty([0,T];H^2(I)^d)$ is a solution to the elastic flow with initial value $z_0$ and given boundary conditions $u_D, u_D'$, $z$ satisfies $z(0) = z_0$, the Euler-Lagrange equation
  \begin{equation}\label{equation:euler_lagrange}
    0 = \int_I z_t \cdot v + z_{xx} \cdot v_{xx} \ dx
  \end{equation}
  for all tangential fields $v$ and the arc-length constraint $\vert z_x \vert^2 = 1$. The arc-length constraint can be incorporated into the Euler Lagrange equation via the use of a Lagrange multiplier. This yields
  \begin{equation}\label{equation:elastic_flow_pde}
    0 = \int_I z_t \cdot v + z_{xx} \cdot v_{xx} + \lambda z_x \cdot v_x \ dx
  \end{equation}
  for all $v \in H^2(I)^d$ satisfying homogeneous boundary conditions, with $\lambda = - z_x \cdot \int_x^b z_t \ d\sigma - \vert z_{xx} \vert^2$ the Lagrange multiplier. This Lagrange multiplier is obtained by testing $(\ref{equation:euler_lagrange})$ with $w = v - \int_a^x (v_x \cdot z_x)z_x \ d\sigma$ for $v$ as above. In its strong form problem $(\ref{equation:elastic_flow_pde})$ reads
  \begin{equation}
    \label{equation:strong problem}
    \begin{gathered}
      z_t + z_{xxxx} - (\lambda z_x)_x = 0 \kurz \text{ in } I \times (0,T), \\
      \begin{aligned}
      z(\cdot,t) &= u_D \text{ on } \Gamma_D \times (0,T), & \leer
      z_x(\cdot,t) &= u_D' \text{ on } \Gamma_D' \times (0,T), \\
      z_{xx} &= 0 \text{ on } (\del I \setminus \Gamma_D') \times (0,T), & \leer
      z_{xxx} - \lambda z_x &= 0 \text{ on } (\del I \setminus \Gamma_D) \times (0,T), \\
      z(\cdot, 0) &= z_0 \text{ in } I, &\leer \vert z_x \vert^2 &= 1 \text{ in } I \times (0,T).
      \end{aligned}
    \end{gathered}
  \end{equation}
  
  Similar problems with elastic flows of curves have already been studied in a variety of different settings. A frequently studied problem is the gradient flow for the energy $\int_\Gamma (\vert \kappa \vert^2/2 + \lambda) \ ds$, where $\kappa$ denotes the curvature vector of the curve $\Gamma$, $\lambda \geq 0$ is a given constant and ds is the arc length element. Numerical schemes for this flow have been proposed and analyzed in~\cite{DKS02, BGN08, DD09, BGN10, BGN12, DN24}. Compared to this, the difference and main difficulty of $(\ref{equation:strong problem})$ lies in the inextensibility constraint $\vert z_x \vert^2 = 1$. A related problem that involves a pointwise constraint on the solution rather than its first order derivative, is the harmonic map heat flow for which a numerical scheme and an error estimate have recently been derived in~\cite{BKW24}.
  A numerical scheme for the approximation of the elastic flow of inextensible curves has been devised in~\cite{Bar13}, see also~\cite{Wal16, BRR18, BR21} for the case of self-avoiding inextensible curves.
    
  The scheme in~\cite{Bar13} uses piecewise cubic $C^1$ functions subject to a partition of $I$ and imposes the inextensibility constraint nodewise, i.e. $\I_{h,1}(\vert z_{hx} \vert^2 - 1) = 0$, where $\I_{h,1}$ is the nodal $\P_1$ interpolant. The time discretization then linearizes this constraint and it is shown in~\cite{Bar13} that the resulting scheme is unconditionally stable and convergent in the sense that every accumulation point of the sequence generated by the scheme solves $(\ref{equation:euler_lagrange})$.
  In this paper we are interested in deriving error estimates for a semi-discrete version of the approach developed in~\cite{Bar13}. In numerical experiments one observes a linear experimental convergence rate for the $H^2$ error, which is suboptimal since the corresponding interpolation error is of quadratic order.
  
  The reason for this suboptimal convergence rate is that the discrete constraint $\I_{h,1}(\vert z_{hx} \vert^2 - 1) = 0$ is too weak. It is a well known property of the nodal $\P_1$ interpolant $\I_{h,1}$ that it minimizes the Dirichlet energy for given values at the nodes, i.e. for all $v \in H^1(I)^d$ we have
  \begin{align*}
    \int_I \vert (\I_{h,1} v)' \vert^2 \ dx \leq \int_I \vert v' \vert^2 \ dx.
  \end{align*}
  Thus, if $u \in H^2(I)^d$ satisfies the discrete arc-length constraint $\vert u'(x_i) \vert^2 = 1$ for all $i = 0,...,M$, with $v = u'$ and $w(x) := \int_a^x \I_{h,1} v \ d\sigma$ we have $E(w) \leq E(u)$ and $w'(x_i) = u'(x_i)$ for all $i = 0,...,M$. Therefore solutions to the discrete minimization problem are piecewise quadratic and the linear convergence rate in $H^2$ is optimal. This can be improved by enforcing the arc-length constraint not just at the endpoints of each subinterval, but also at their midpoints, i.e. requiring that $\I_{h,2}(\vert z_{hx} \vert^2 - 1) = 0$ where $\I_{h,2}$ is the nodal $\P_2$-interpolant.
  The goal of this paper is to derive a quasi-optimal error estimate for a semi-discrete gradient flow using this improved discrete constraint and to verify the results using numerical simulations.
  
  % Notation
  \subsection{Notation}
  The following notation will be used throughout this paper. Let $\overline I = \bigcup_{i = 1}^M [x_{i-1}, x_i]$ be a dissection of the interval $I = (a,b) \subset \real$ with $a = x_0 < x_1 < ... < x_M = b$. We set $I_i := [x_{i-1}, x_i]$, $h_i = x_i - x_{i-1}$, $h = \max_i h_i$, $\T_h = \{I_i \ \vert \ i = 1,...,M\}$ and assume that there exists $c > 0$ such that $h \leq c h_i$ for $i = 1,...,M$. We then define the finite element spaces
  \begin{align*}
    \S^{k,l}(\T_h) := \{v_h \in C^l(\overline I) \ \vert \ v_h\vert_J \in \P_k \text{ for all } J \in \T_h \} \subset H^{l+1}(I).
  \end{align*}
  To deal with boundary values, for $\Gamma_D, \Gamma_D' \subset \del I$ we also define the Sobolev spaces with vanishing boundary values
  \begin{gather*}
    H^2_D(I) := \{ v \in H^2(I)\ \vert \ v\vert_{\Gamma_D} = 0,\ 
    v'\vert_{\Gamma_D'} = 0\}, \\
    H^1_D(I) := \{ v \in H^1(I)\ \vert \ v\vert_{\Gamma_D} = 0 \}, \leer
    H^1_{D'}(I) := \{ v \in H^1(I)\ \vert \ v\vert_{\Gamma_D'} = 0 \}.
  \end{gather*}
  Analogously for the finite element spaces we set
  \begin{gather*}
    \S^{k,l}_D(\T_h) := \S^{k,l}(\T_h) \cap H^{l+1}_D(I).
  \end{gather*}
  We write $(\cdot , \cdot)$ and $\Vert \cdot \Vert$ for the $L^2$--product and norm and $D_h u$ for the elementwise weak derivative of a function $u$. Also for $i = 1,...,M$ we set $m_i := (x_{i-1} + x_i)/2$ the midpoint of the interval $I_i$ and define $\M_h(\T_h) := \{m_i\ \vert\ i=1,...,M\}$ as well as
  \begin{align*}
    \N_1(\T_h) := \{x_i \ \vert \ i = 0,...,M\}, \leer
    \N_2(\T_h) := \N_1(\T_h) \cup \M_h(\T_h),
  \end{align*}
  the sets of associated nodes for $\S^{1,0}(\T_h)$ and $\S^{2,0}(\T_h)$.
  We then define the cubic $C^1$ interpolant $\I_{h,3}: C^1(\overline I)^d \to \S^{3,1}(\T_h)^d$ and the continuous quadratic and linear interpolants $\I_{h,2}: C^0(\overline I)^d \to \S^{2,0}(\T_h)^d$, $\I_{h,1}: C^0(\overline I)^d \to \S^{1,0}(\T_h)^d$ via the identities
  \begin{gather*}
    \I_{h,3}v(z) = v(z),\ (\I_{h,3}v)'(z) = v'(z) 
    \kurz \text{for all } z \in \N_1(\T_h), \\
    \I_{h,2}v(z) = v(z)
    \kurz \text{for all } z \in \N_2(\T_h), \\
    \I_{h,1}v(z) = v(z) 
    \kurz \text{for all } z \in \N_1(\T_h).
  \end{gather*}
  One important property of $\I_{h,3}$ is that for $\Gamma_D, \Gamma_D' \neq \emptyset$ it defines an orthogonal projection from $H^2_D(I)^d$ onto $\S^{3,1}_D(\T_h)^d$ with respect to the scalar product $(v,w)_{H^2_D(I)^d} := \int_I v'' \cdot w'' \ dx$, see Lemma \ref{lemma:rhs_interpolation}. Further, we introduce another interpolant $\J_{h,3}: C^1(\overline I)^d \to \S^{3,1}(I)^d$ defined via
  \begin{equation*}
    \J_{h,3}v(x) = v(a) + \int_a^x \I_{h,2}v' \ d\sigma.
  \end{equation*}
  We note that, according to Lemma \ref{lemma:interpolation_estimate}, $\J_{h,3}$ satisfies essentially the same interpolation estimate as $\I_{h,3}$, although under slightly stricter regularity conditions.
  The crucial advantage of $\J_{h,3}$ is that it preserves the values of $v'$ not just at the endpoints of each subinterval, but also at the midpoints. The disadvantage of this interpolant is that it does not preserve boundary values at the endpoint $b$ of the interval, i.e. in general we have $(\J_{h,3}v)(b) \neq v(b)$. This is also one of the reasons why the case $\Gamma_D = \del I$ is excluded from the error estimate. 
  
  % error estimate
  \section{Error estimate}
  \label{section:error_estimate_flow}
  In this section we derive an error estimate for the semi-discrete elastic flow.
  For this we first linearize the arc-length constraint. We note that a function $z \in H^1((0,T);C^1(I)^d)$ satisfies the arc-length constraint $\vert z_x(t) \vert^2 = 1$ for all $t \in (0,T)$ if and only if
  \begin{align*}
    \vert z_x(0) \vert^2 = 1, \leer 0 = \frac{1}{2} \frac{d}{dt} \vert z_x \vert^2
    = z_{tx} \cdot z_x.
  \end{align*}
  Analogously a function $z_h \in H^1((0,T);\S^{3,1}(\T_h)^d)$ satisfies the discrete arc-length constraint $\I_{h,2}(\vert z_x \vert^2-1) = 0$ if and only if
  \begin{align*}
    \I_{h,2}(\vert z_{hx}(0) \vert^2 - 1) = 0, \leer
    0 = \frac{1}{2} \frac{d}{dt} \I_{h,2}(\vert z_{hx} \vert^2 - 1)
    = \I_{h,2}(z_{htx} \cdot z_{hx}).
  \end{align*}
  Thus for $z \in H^2(I)^d$ and $z_h \in \S^{3,1}(\T_h)^d$ we set
  \begin{align*}
    \G(z) := \{v \in H^2_D(I)^d \ \vert \ z' \cdot v' = 0 \text{ in } I \}, \kurz
    \G_h(z_h) := \{ v_h \in \S^{3,1}_D(\T_h)^d \ \vert \ \I_{h,2}(z_h' \cdot v_h') = 0 \text{ in } I \}.
  \end{align*}
  This allows us to reformulate the definition of the elastic flow: A function $z \in H^1((0,T);L^2(I)^d) \cap L^\infty([0,T];H^2(I)^d)$ is a solution to the elastic flow $(\ref{equation:elastic_flow_pde})$ if and only if $z$ satisfies $z(0) = z_0$, $z_t(t) \in \G(z(t))$ for almost all $t \in (0,T)$ and
  \begin{equation*}
    0 = \int_I z_t \cdot y + z_{xx} \cdot y_{xx} \ dx
  \end{equation*}
  for all $y \in \G(z(t))$ and almost all $t \in (0,T)$. Analogously we now define the semi-discrete bending problem: We call a function $z_h \in H^1((0,T);\S^{3,1}(\T_h)^d)$ a solution to the semi-discrete elastic flow if and only if $z_h$ satisfies $z_h(0) = \J_{h,3} (z_0),\ 
  z_{ht}(t) \in \G_h(z_h(t))$ for almost all $t \in (0,T)$ and 
  \begin{equation}\label{equation:semi_discrete_scheme}
    \int_I z_{ht}(t) \cdot y_h + z_{hxx}(t) \cdot y_{hxx} \ dx = 0
  \end{equation}
  for all $y_h \in \G_h(z_h(t))$ and almost all $t \in (0,T)$. Note that the conditions $z_t(t) \in \G(z(t))$ and $z_{ht}(t) \in \G_h(z_h(t))$ also include the boundary conditions.
  For now we will just assume, that the semi-discrete problem has a solution satisfying 
  \begin{equation}
    \label{equation:discrete_boundedness}
    \max_{t \in [0,T]} \max_{i = 1,...,M} \Vert z_h(t) \Vert_{W^{3,\infty}(I_i)} \leq c,
  \end{equation}
  where $c$ is independent of $h$.
  A justification for this assumption will be given later on.\\
  The crucial step to obtain an error estimate is to construct suitable test functions for both, the continuous and the semi-discrete problems. For the corresponding linear problem, the standard approach is to test the continuous problem with the approximation error $z_t - z_{ht}$ and the discrete problem with its interpolant $\I_{h,3} z_t - z_{ht}$. This however does not work in this case as neither of these functions satisfies the required constraints. To still be able to test with these approximation errors we introduce the following correction terms
  \begin{align*}
    \delta(x,t) &:= \int_a^x ((z_{tx} - z_{htx}) \cdot z_x ) z_x \ d\sigma
    = - \int_a^x (z_{htx} \cdot z_x) z_x \ d\sigma, \\
    \delta_h(x,t) &:= \int_a^x \I_{h,2}((((\I_{h,3} z_t)_x - z_{htx}) \cdot z_{hx}) z_{hx}) \ d\sigma
    = \int_a^x \I_{h,2}(((\I_{h,3} z_t)_x \cdot z_{hx}) z_{hx}) \ d\sigma,
  \end{align*}
  and set $y := z_t - z_{ht} - \delta \in H^2_D(I)^d$, $y_h := \I_{h,3} z_t - z_{ht} - \delta_h \in \S^{3,1}_D(\T_h)^d$.
  Since $\vert z_x \vert^2 = 1$ we have
  \begin{align*}
    y_x \cdot z_x
    = (z_{tx} - z_{htx}) \cdot z_x - \delta_x \cdot z_x
    = (z_{tx} - z_{htx}) \cdot z_x 
    - ((z_{tx} - z_{htx}) \cdot z_x) \vert z_x \vert^2
    = 0.
  \end{align*}
  Similarly, since $\I_{h,2}(\vert z_{hx} \vert^2) = 1$ we have
  \begin{align*}
    \I_{h,2}(y_{hx} \cdot z_{hx})
    &= \I_{h,2}((\I_{h,3} z_t - z_{ht})_x \cdot z_{hx} 
    - \delta_{hx} \cdot z_{hx}) \\
    &= \I_{h,2}((\I_{h,3} z_t - z_{ht})_x \cdot z_{hx} 
    - \I_{h,2}((((\I_{h,3} z_t)_x - z_{htx}) \cdot z_{hx}) z_{hx}) 
    \cdot z_{hx}) \\
    &= \I_{h,2}( (\I_{h,3} z_t - z_{ht})_x \cdot z_{hx} - ((\I_{h,3} z_t - 
    z_{htx})_x \cdot z_{hx}) \I_{h,2} \vert z_{hx} \vert^2) \\
    &= 0.
  \end{align*}
  Therefore $y(t) \in \G(z(t))$ and $y_h(t) \in \G_h(z_h(t))$ are admissible test functions for the continuous and semi-discrete problem, respectively. We further set
  \begin{align*}
    \delta^h(x,t) &:= \int_a^x ((\I_{h,3} z_t)_x \cdot z_{hx}) z_{hx} \ d\sigma.
  \end{align*}
  Next we show some crucial properties of $\delta^h$ and $\delta_h$.
  
  \begin{lemma}\label{lemma:delta_estimates}
    Assume that $z \in L^\infty((0,T);H^2(I)^d)$ with $z_t \in L^\infty((0,T);H^4(I)^d)$ and that $z_h$ satisfies $(\ref{equation:discrete_boundedness})$.
    The functions $\delta_h$ and $\delta^h$ then satisfy for all $t \in [0,T]$
    \begin{equation}\label{equation:delta_difference}
      \max_{x \in I} \vert \delta_h(x,t) - \delta^h(x,t) \vert \leq c h^4.
    \end{equation}
    Further for $\delta^h$ we have for all $t \in [0,T]$
    \begin{align}
      \label{equation:delta_sup}
      \Vert \delta^h(t) \Vert_{L^\infty(I)^d} 
      &\leq ch^3 + c \Vert z_x(t) - z_{hx}(t) \Vert, \\
      \label{equation:delta_derivative}
      \Vert \delta^h_{xx}(t) \Vert
      &\leq c h^2 + c \Vert z(t) - z_h(t) \Vert_{H^2(I)^d}.
    \end{align}
  \end{lemma}
  
  \begin{proof}
    According to Lemma \ref{lemma:new simpson rule} we have for all $i \in \{0,...,M\}$
    \begin{align*}
      \vert \delta_h(x_i) - \delta^h(x_i) \vert
      &= \left\vert \int_a^{x_i} ((\I_{h,3} z_t)_x \cdot z_{hx}) z_{hx} \ d\sigma
      - \int_a^{x_i} \I_{h,2}(((\I_{h,3} z_t)_x 
      \cdot z_{hx}) z_{hx}) \ d\sigma \right\vert \\
      &\leq c h^4 \Vert D_h^4(((\I_{h,3} z_t)_x \cdot z_{hx}) z_{hx}) \Vert_{L^\infty(I)^d}
      \leq c h^4.
    \end{align*}
    In the last estimate we have also used the assumption $(\ref{equation:discrete_boundedness})$.
    Therefore for $x \in I_i$ we obtain
    \begin{align*}
      \vert \delta_h(x) - \delta^h(x) \vert
      &\leq \vert \delta_h(x_{i-1}) - \delta^h(x_{i-1}) \vert
      + c h \max_{I_i} \vert \delta_{hx} - \delta^h_x \vert \\
      &\leq c h^4 + c h \max_{I_i} \vert ((\I_{h,3} z_t)_x \cdot z_{hx}) z_{hx}
      - \I_{h,2}(((\I_{h,3} z_t)_x \cdot z_{hx}) z_{hx}) \vert.
    \end{align*}
    Now the interpolation estimate from Lemma \ref{lemma:interpolation_estimate} and $(\ref{equation:discrete_boundedness})$ yield $(\ref{equation:delta_difference})$.
    Let now $x \in I$ arbitrary. Using $z_{tx} \cdot z_x = 0$ we obtain
    \begin{align*}
      \vert \delta^h(x) \vert
      &\leq c \int_I \vert (\I_{h,3} z_t)_x \cdot z_{hx} \vert \ d\sigma
      \leq c \int_I \vert ((\I_{h,3} z_t)_x - z_{tx}) \cdot z_{hx} \vert
      + \vert z_{tx} \cdot (z_{hx} - z_x) \vert \ d\sigma \\
      &\leq c \Vert (\I_{h,3} z_t)_x - z_{tx} \Vert
      + c \Vert z_x - z_{hx} \Vert
      \leq c h^3 + c \Vert z_x - z_{hx} \Vert.
    \end{align*}
    This proves (\ref{equation:delta_sup}).
    To prove the last inequality we calculate
    \begin{align*}
      \delta^h_{xx}
      &= ((\I_{h,3} z_t)_{xx} \cdot z_{hx}) z_{hx}
      + ((\I_{h,3} z_t)_x \cdot z_{hxx}) z_{hx}
      + ((\I_{h,3} z_t)_x \cdot z_{hx}) z_{hxx} \\
      &= (((\I_{h,3} z_t)_{xx} - z_{txx}) \cdot z_{hx}) z_{hx}
      + (z_{txx} \cdot (z_{hx} - z_x)) z_{hx} + (z_{txx} \cdot z_x) z_{hx} \\
      &+ (((\I_{h,3} z_t)_x - z_{tx}) \cdot z_{hxx}) z_{hx}
      + (z_{tx} \cdot (z_{hxx} - z_{xx})) z_{hx} + (z_{tx} \cdot z_{xx}) z_{hx} \\
      &+ (((\I_{h,3} z_t)_x - z_{tx}) \cdot z_{hx}) z_{hxx}
      + (z_{tx} \cdot (z_{hx} - z_x)) z_{hxx}.
    \end{align*}
    Using $z_{txx} \cdot z_x + z_{tx} \cdot z_{xx} = \del_x (z_{tx} \cdot z_x) 
    = 0$ and Hölder's inequality, we obtain
    \begin{align*}
      \Vert \delta^h_{xx} \Vert_{L^2(I)^d}
      &\leq \Vert (\I_{h,3} z_t)_{xx} - z_{txx} \Vert 
      \Vert z_{hx} \Vert_{L^\infty(I)^d}^2
      + \Vert z_{txx} \Vert \Vert z_{hx} - z_x \Vert_{L^\infty(I)^d} 
      \Vert z_{hx} \Vert_{L^\infty(I)^d} \\
      &+ \Vert (\I_{h,3} z_t)_x - z_{tx} \Vert_{L^\infty(I)^d}
      \Vert z_{hxx} \Vert \Vert z_{hx} \Vert_{L^\infty(I)^d}
      + \Vert z_{tx} \Vert_{L^\infty(I)^d} \Vert z_{hxx} - z_{xx} \Vert
      \Vert z_{hx} \Vert_{L^\infty(I)^d} \\
      &+ \Vert (\I_{h,3} z_t)_x - z_{tx} \Vert_{L^\infty(I)^d}
      \Vert z_{hx} \Vert_{L^\infty(I)^d} \Vert z_{hxx} \Vert
      + \Vert z_{tx} \Vert_{L^\infty(I)^d} \Vert z_{hx} - z_{x} \Vert_{L^\infty(I)^d}
      \Vert z_{hxx} \Vert \\
      &\leq c \Vert z_t - \I_{h,3} z_t \Vert_{H^2(I)^d} 
      + c \Vert z - z_h \Vert_{H^2(I)^d}.
    \end{align*}
    For the last estimate we have also used the continuity of the embedding $H^1(I)^d \embeds C^0(I)^d$ as well as $(\ref{equation:discrete_boundedness})$.
    Finally an interpolation estimate proves
    \begin{align*}
      \Vert \delta^h_{xx} \Vert \leq c h^2 + c \Vert z - z_h \Vert_{H^2(I)^d}
    \end{align*}
    which proves (\ref{equation:delta_derivative}).
  \end{proof}
  
  We are now able to bound the approximation error of the semi-discrete scheme in $H^1([0,T];L^2(I)^d) \cap L^\infty([0,T];H^2(I)^d)$.
  
  \begin{theorem}[error estimate]
    \label{theorem:error_estimate}
    Let $z \in C^0([0,T];H^4(I)^d)$ be a solution of the continuous elastic flow $(\ref{equation:elastic_flow_pde})$ with $z_t \in L^\infty((0,T); H^4(I)^d)$. Further assume that the Lagrange multiplier $$\lambda = -z_x \cdot \int_x^b z_t \ d\sigma - \vert z_{xx} \vert^2$$ satisfies $\lambda \in W^{1,\infty}((0,T);W^{1,1}(I))$. Lastly let $z_h$ be a solution to the semi-discrete scheme $(\ref{equation:semi_discrete_scheme})$ that satisfies $(\ref{equation:discrete_boundedness})$. Then there exists $h_0 > 0$ such that for all $0 < h \leq h_0$ we have the error estimate
    \begin{equation}
      \label{equation:error_estimate}
      \Vert z_t - z_{ht} \Vert_{L^2([0,T];L^2(I)^d)}^2
      + \Vert z - z_h \Vert_{L^\infty([0,T];H^2(I)^d)}^2
      \leq c h^4
    \end{equation}
    with a constant $c$ that is independent of $h$.
  \end{theorem}
  
  \begin{proof}
    By definition, $z_t - z_{ht}$ and $\I_{h,3} z_t - z_{ht}$ satisfy
    \begin{align*}
      z_t - z_{ht} = y + \delta, \leer \I_{h,3} z_t - z_{ht} = y_h + \delta_h.
    \end{align*}
    We therefore get
    \begin{gather*}
      \int_I \vert z_t - z_{ht} \vert^2 \ dx
      + \frac{1}{2} \frac{d}{dt} \int_I \vert z_{xx} - z_{hxx} \vert^2 \ dx
      = \int_I z_t \cdot (y + \delta) + z_{xx} \cdot (y_{xx} + \delta_{xx}) \ dx\\
      \kurz - \int_I z_{ht} \cdot (z_t - \I_{h,3} z_t + y_h + \delta_h)
      + z_{hxx} \cdot (z_{txx} - (\I_{h,3} z_t)_{xx} + y_{hxx} + \delta_{hxx}) \ dx.
    \end{gather*}
    Using $\int_I z_t \cdot y + z_{xx} \cdot y_{xx} = 0$, $\int_I z_{ht} \cdot y_h
    + z_{hxx} \cdot y_{hxx} = 0$ and Lemma \ref{lemma:rhs_interpolation} we obtain
    \begin{equation}\label{equation:I}
      \begin{aligned}
        &\int_I \vert z_t - z_{ht} \vert^2 \ dx
        + \frac{1}{2} \frac{d}{dt} \int_I \vert z_{xx} - z_{hxx} \vert^2 \ dx \\
        &\kurz = \int_I z_t \cdot \delta + z_{xx} \cdot \delta_{xx} \ dx
        - \int_I z_{ht} \cdot \delta_h + z_{hxx} \cdot \delta_{hxx} \ dx \\
        &\kurz - \int_I z_t \cdot (z_t - \I_{h,3} z_t)
        + (z_{ht} - z_t) \cdot (z_t - \I_{h,3} z_t) \ dx.
      \end{aligned}
    \end{equation}
    With Hölder's inequality and the $\varepsilon$-Young-inequality we 
    can estimate
    \begin{equation}\label{equation:other_terms}
      \begin{aligned}
        \left\vert \int_I z_t \cdot (z_t - \I_{h,3} z_t) \ dx \right\vert
        &\leq c h^4, \\
        \left\vert \int_I (z_{ht} - z_t) \cdot (z_t - \I_{h,3} z_t) \ dx \right\vert
        &\leq \varepsilon \Vert z_t - z_{ht} \Vert^2 + c_\varepsilon h^8.
      \end{aligned}
    \end{equation}
    Also with the help of Lemma \ref{lemma:delta_estimates} we can estimate
    \begin{equation}\label{equation:time_estimate}
      \begin{aligned}
        - \int_I z_{ht} \cdot \delta_h \ dx
        &= - \int_I (z_t +z_{ht} - z_t) \cdot (\delta_h - \delta^h) \ dx
        - \int_I (z_{ht} - z_t) \cdot \delta^h \ dx - \int_I z_t \cdot \delta^h \ dx \\
        &\leq (\Vert z_t \Vert + \Vert z_t - z_{ht} \Vert) 
        \Vert \delta_h - \delta^h \Vert
        + \Vert z_t - z_{ht} \Vert \Vert \delta^h \Vert 
        - \int_I z_t \cdot \delta^h \ dx \\
        &\leq c h^4 + \varepsilon \Vert z_t - z_{ht} \Vert^2 + c_\varepsilon h^8
        + c_\varepsilon h^4 + c_\varepsilon \Vert z_x - z_{hx} \Vert^2 
        - \int_I z_t \cdot \delta^h \ dx\\
        &= c_\varepsilon h^4 + \varepsilon \Vert z_t - z_{ht} \Vert^2 
        + c_\varepsilon \Vert z_x - z_{hx} \Vert^2
        - \int_I z_t \cdot \delta^h \ dx.
      \end{aligned}
    \end{equation}
    For the last remaining term in (\ref{equation:I}) we get
    \begin{equation}\label{equation:estimate_S}
      \begin{aligned}
        - \int_I z_{hxx} \cdot \delta_{hxx} \ dx
        &= \int_I (z_{xx} - z_{hxx}) \cdot (\delta_{hxx} - \delta^h_{xx}) \ dx
        + \int_I (z_{xx} - z_{hxx}) \cdot \delta^h_{xx} \ dx\\
        &\kurz - \int_I z_{xx} \cdot (\delta_{hxx} - \delta^h_{xx}) \ dx
        - \int_I z_{xx} \cdot \delta^h_{xx} \ dx \\
        &=: S_1 + S_2 + S_3 - \int_I z_{xx} \cdot \delta^h_{xx} \ dx.
      \end{aligned}
    \end{equation}
    $S_1$ and $S_2$ we can easily estimate using Hölder's inequality, Lemma
    \ref{lemma:delta_estimates} and $(\ref{equation:discrete_boundedness})$:
    \begin{align*}
      \vert S_1 \vert 
      &\leq \Vert z_{hxx} - z_{xx} \Vert \Vert \delta_{hxx} - \delta^h_{xx} \Vert \\
      &= \Vert z_{xx} - z_{hxx} \Vert
      \Vert (((\I_{h,3} z_t)_x \cdot z_{hx}) z_{hx} 
      - \I_{h,2}(((\I_{h,3} z_t)_x \cdot z_{hx}) z_{hx}))_x \Vert \\
      &\leq c h^2 \Vert z_{xx} - z_{hxx} \Vert 
      \leq c h^4 + c \Vert z - z_h \Vert_{H^2(I)^d}^2, \\
      \vert S_2 \vert
      &\leq \Vert z_{hxx} - z_{xx} \Vert \Vert \delta^h_{xx} \Vert
      \leq c h^2 \Vert z_{xx} - z_{hxx} \Vert
      + c \Vert z - z_h \Vert_{H^2(I)^d}^2
      \leq c h^4 + c \Vert z - z_h \Vert_{H^2(I)^d}^2.
    \end{align*}
    To get an estimate for $S_3$ we first note that by definition we have $\delta_{hx} = \I_{h,2}(\delta^h_x)$, thus we have $\delta_{hx} = \delta^h_x$ on $\del I$ and integration by parts yields
    \begin{align*}
      S_3 
      &= - \int_I z_{xx} \cdot (\delta_{hxx} - \delta^h_{xx}) \ dx
      = \int_I z_{xxx} \cdot (\delta_{hx} - \delta^h_{x}) \ dx \\
      &= -\int_I z_{xxx} \cdot (((\I_{h,3} z_t)_x \cdot z_{hx}) z_{hx}
      - \I_{h,2}(((\I_{h,3} z_t)_x \cdot z_{hx}) z_{hx})) \ dx.
    \end{align*}
    We now apply Lemma \ref{lemma:new simpson rule} to obtain
    \begin{align*}
      S_3
      &\leq c h^4 \Vert z_{xxx} \Vert_{L^1(I)^d} \Vert D_h^4 (((\I_{h,3} z_t)_x \cdot z_{hx})z_{hx}) \Vert_{L^\infty(I)^d} \\
      &+ c h^4\Vert z_{xxxx} \Vert_{L^1(I)^d} \Vert D_h^3(((\I_{h,3} z_t)_x \cdot z_{hx}) z_{hx}) \Vert_{L^\infty(I)^d} \\
      &\leq c h^4.
    \end{align*}
    Inserting those estimates into (\ref{equation:estimate_S}) yields
    \begin{equation}\label{equation:space_estimate}
      - \int_I z_{hxx} \cdot \delta_{hxx} \ dx
      \leq c h^4 + c \Vert z - z_h \Vert_{H^2(I)^d}^2 
      - \int_I z_{xx} \cdot \delta^h_{xx} \ dx.
    \end{equation}
    Combining $(\ref{equation:I}) - (\ref{equation:space_estimate})$ results in
    \begin{align*}
      &\int_I \vert z_t - z_{ht} \vert^2 \ dx
      + \frac{1}{2} \frac{d}{dt} \int_I \vert z_{xx} - z_{hxx} \vert^2 \ dx  \\
      &\kurz \leq \int_I z_t \cdot (\delta - \delta^h) 
      + z_{xx} \cdot (\delta_{xx} - \delta^h_{xx}) \ dx
      + c_\varepsilon h^4 + 2 \varepsilon \Vert z_t - z_{ht} \Vert^2
      + c_\varepsilon \Vert z - z_h \Vert_{H^2(I)^d}^2.
    \end{align*}
    By definition $\delta - \delta^h$ satisfies the boundary conditions
    \begin{align*}
      \delta - \delta^h = 0 \text{ on } \Gamma_D,\kurz
      \delta_x - \delta^h_x = 0 \text{ on } \Gamma_D'
    \end{align*}
    and thus $\delta - \delta^h \in H^2_D(I)^d$. Therefore $(\ref{equation:elastic_flow_pde})$ yields
    \begin{align*}
      \int_I z_t \cdot (\delta - \delta^h) + z_{xx} \cdot (\delta_{xx} - \delta^h_{xx}) \ dx = - \int_I \lambda z_x \cdot (\delta_x - \delta^h_x) \ dx.
    \end{align*}
    We now choose $\varepsilon = \frac{1}{4}$ and obtain
    \begin{equation}
      \label{equation:estimate lambda}
      \frac{1}{2} \Vert z_t - z_{ht} \Vert^2 + \frac{1}{2} \frac{d}{dt}
      \Vert z_{xx} - z_{hxx} \Vert^2
      \leq c h^4 + c \Vert z - z_h \Vert_{H^2(I)^d}^2
      - \int_I \lambda z_x \cdot (\delta_x - \delta^h_x) \ dx.
    \end{equation}
    To deal with the integral term we calculate
    \begin{equation}\label{equation:integral_term_I}
      -\int_I \lambda z_x \cdot \delta_x \ dx = \int_I \lambda z_{htx} \cdot z_x \ dx.
    \end{equation}
    Further we get
    \begin{align*}
      \int_I \lambda z_x \cdot \delta^h_x \ dx
      &= \int_I \lambda z_x \cdot (((\I_{h,3} z_t)_x \cdot z_{hx}) z_{hx}) \ dx \\
      &= \int_I \lambda (((\I_{h,3} z_t)_x - z_{tx}) \cdot z_{hx}) 
      (z_{hx} \cdot z_x) \ dx 
      + \int_I \lambda (z_{tx} \cdot z_{hx}) (z_{hx} \cdot z_x) \ dx. 
    \end{align*}
    We use $(\I_{h,3} z_t)_x - z_{tx} = 0$ on $\del I$ and integrate by parts to
    obtain
    \begin{align*}
      \int_I \lambda z_x \cdot \delta^h_x \ dx
      &= -\int_I (\I_{h,3} z_t - z_t) \cdot (\lambda z_{hx} (z_{hx} \cdot z_x))_x 
      \ dx 
      + \int_I \lambda (z_{tx} \cdot z_{hx}) \ dx \\
      &\kurz + \int_I \lambda (z_{tx} \cdot z_{hx}) (z_{hx} \cdot z_x - 1) \ dx 
      \\
      &\leq c h^4 + \int_I \lambda (z_{tx} \cdot z_{hx}) \ dx
      + \int_I \lambda (z_{tx} \cdot (z_{hx} - z_x))((z_{hx} - z_x) \cdot 
      z_x) \ dx.
    \end{align*}
    Hölder's inequality then implies
    \begin{equation}\label{equation:integral_term_II}
      \int_I \lambda z_x \cdot \delta^h_x \ dx
      \leq c h^4 + c \Vert z - z_h \Vert_{H^2(I)^d}^2 
      + \int_I \lambda (z_{tx} \cdot z_{hx}) \ dx.
    \end{equation}
    Combining (\ref{equation:integral_term_I}) and 
    (\ref{equation:integral_term_II}) with (\ref{equation:estimate lambda}) yields
    \begin{equation}\label{equation:formula_before_integration}
      \begin{aligned}
        &\frac{1}{2} \Vert z_t - z_{ht} \Vert^2 + \frac{1}{2} \frac{d}{dt}
        \Vert z_{xx} - z_{hxx} \Vert^2
        \leq c h^4 + c \Vert z - z_h \Vert_{H^2(I)^d}^2 
        + \int_I \lambda \del_t(z_x \cdot z_{hx} - 1) \\
        &\leer = ch^4 + c \Vert z - z_h \Vert_{H^2(I)^d}^2
        + \frac{d}{dt} \int_I \lambda (z_x \cdot z_{hx} - 1) \ dx
        - \int_I \lambda_t (z_x \cdot z_{hx} - 1) \ dx.
      \end{aligned}
    \end{equation}
    Integrating (\ref{equation:formula_before_integration}) over $(0,t)$ therefore
    yields
    \begin{equation}\label{equation:formula_after_integration}
      \begin{aligned}
        &\frac{1}{2} \int_0^t \Vert z_t - z_{ht} \Vert^2 \ ds
        + \frac{1}{2}( \Vert (z_{xx} - z_{hxx})(t) \Vert^2 
        - \Vert (z_{xx} - z_{hxx})(0) \Vert^2) \\
        &\leer \leq c h^4 + \int_I (\lambda (z_x \cdot z_{hx} - 1))(t) 
        - (\lambda(z_x \cdot z_{hx} - 1))(0) \ dx\\
        &\leer \kurz + c \int_0^t \Vert z - z_h \Vert_{H^2(I)^d}^2 \ ds
        - \int_0^t \int_I \lambda_t (z_x \cdot z_{hx} - 1) \ dx\ ds.
      \end{aligned}
    \end{equation}
    By definition we have $\I_{h,2} (\vert z_{hx} \vert^2) = 1
    = \vert z_x \vert^2$. We therefore get
    \begin{align*}
      z_x \cdot z_{hx} - 1
      &= z_x \cdot z_{hx} - \frac{1}{2} \vert z_x \vert^2 - \frac{1}{2} \vert z_{hx} \vert^2 + \frac{1}{2} \vert z_{hx} \vert^2 - \frac{1}{2} \\
      &= -\frac{1}{2} \vert z_x - z_{hx} \vert^2 + \frac{1}{2} \left(\vert z_{hx} \vert^2 - \I_{h,2} \vert z_{hx} \vert^2 \right).
    \end{align*}
    Lemma \ref{lemma:new simpson rule}, the stability of the interpolant $\I_{h,2}$ and Lemma \ref{corollary:GN} then imply
    \begin{equation}\label{equation:first_lambda_estimate}
      \begin{aligned}
        &\int_I (\lambda (z_x \cdot z_{hx} - 1))(t) \ dx
        = -\frac{1}{2} \int_I \lambda(t) \vert z_x(t) - z_{hx}(t) \vert^2
        - \lambda(t) (\vert z_{hx} \vert^2 - \I_{h,2} \vert z_{hx} \vert^2 )(t) \ dx \\
        &\kurz \leq c h^4 (\Vert \lambda(t) \Vert_{L^1(I)} \Vert D_h^4 \vert z_{hx}(t) \vert^2 \Vert_{L^\infty(I)} + \Vert \lambda_x(t) \Vert_{L^1(I)} \Vert D_h^3 \vert z_{hx}(t) \vert^2 \Vert_{L^\infty(I)}) \\
        &\leer + \frac{1}{2} \Vert \lambda(t) \Vert_{L^1(I)} \Vert z_x(t) - z_{hx}(t) \Vert_{L^\infty(I)^d}^2 \\
        &\kurz \leq c h^4 + \frac{1}{4} \Vert z_{xx}(t) - z_{hxx}(t) \Vert^2 + c \Vert z(t) - z_h(t) \Vert^2.
      \end{aligned}
    \end{equation}
    For $t = 0$ we therefore get:
    \begin{equation}\label{equation:second_lambda_estimate}
      \begin{aligned}
        \int_I (\lambda (z_x \cdot z_{hx} - 1))(0) \ dx
        &\leq c h^4.
      \end{aligned}
    \end{equation}
    Analogously we obtain for almost all $t \in [0,T]$
    \begin{equation}\label{equation:lambda_t_estimate}
      \begin{aligned}
        &\int_I (\lambda_t (z_x \cdot z_{hx} - 1))(t) \ dx
        = -\frac{1}{2} \int_I \lambda_t(t) \vert z_x(t) - z_{hx}(t) \vert^2
        - \lambda_t(t) (\vert z_{hx} \vert^2 - \I_{h,2} \vert z_{hx} \vert^2 )(t) \ dx \\
        &\kurz \leq c h^4 (\Vert \lambda_t(t) \Vert_{L^1(I)} \Vert D_h^4 \vert z_{hx}(t) \vert^2 \Vert_{L^\infty(I)} + \Vert \lambda_{tx}(t) \Vert_{L^1(I)} \Vert D_h^3 \vert z_{hx}(t) \vert^2 \Vert_{L^\infty(I)}) \\
        &\leer + \frac{1}{2} \Vert \lambda_t(t) \Vert_{L^1(I)} \Vert z_x(t) - z_{hx}(t) \Vert_{L^\infty(I)^d}^2 \\
        &\kurz \leq c h^4 + c \Vert z(t) - z_h(t) \Vert_{H^2(I)^d}^2
      \end{aligned}
    \end{equation}
    Also we have the estimate
    \begin{equation}\label{equation:z_h_L2_estimate}
      \begin{aligned}
        c \Vert z(t) - z_h(t) \Vert^2 
        &= c \Vert z(0) - z_h(0) \Vert^2 + c \int_0^t \frac{d}{dt} \Vert z(s) - z_h(s) \Vert^2 \ ds \\
        &\leq c h^8 + \frac{1}{4} \int_0^t \Vert z_t - z_{ht} \Vert^2 \ dx
        + c \int_0^t \Vert z - z_h \Vert^2 \ ds.
      \end{aligned}
    \end{equation}
    Combining all the estimates (\ref{equation:formula_after_integration})
    - (\ref{equation:z_h_L2_estimate}) and Lemma $(\ref{lemma:gagliardo_nirenberg})$ yields
    \begin{equation}
      \label{equation:final term before grönwall}
      \frac{1}{4} \int_0^t \Vert z_t - z_{ht} \Vert^2 \ ds
      + \frac{1}{4} \Vert z(t) - z_h(t) \Vert_{H^2(I)^d}^2 
      \leq c h^4 + c \int_0^t \Vert z - z_h \Vert_{H^2(I)^d}^2 \ ds.
    \end{equation}
    With $u(t) := \Vert z(t) - z_h(t) \Vert_{H^2(I)^d}^2$, Grönwall's inequality yields
    \begin{equation*}
      \begin{aligned}
        &\Vert z(t) - z_h(t) \Vert_{H^2(I)^d}^2
        = u(t) \leq c h^4 \exp \left(\int_0^t c \ ds \right)
        = c h^4
      \end{aligned}
    \end{equation*}
    and therefore $(\ref{equation:final term before grönwall})$ implies
    \begin{align*}
      \int_0^t \Vert z_t - z_{ht} \Vert^2 \ dt + \Vert z(t) - z_{ht}(t) \Vert_{H^2(I)^d}^2 \leq c h^4.
    \end{align*}
    Finally taking the supremum over all $t$ yields the asserted estimate
    (\ref{equation:error_estimate}).
  \end{proof}
  
  Next we deal with the assumption 
  $\Vert z_h \Vert_{W^{3,\infty}(I_i)} \leq c$,
  that we made in Theorem \ref{theorem:error_estimate}.
  
  \begin{lemma}\label{lemma:regularity_estimate}
    Assume $z \in C^0([0,T];H^4(I)^d)$ and let $z_h$ be a solution to the semi-discrete gradient flow
    $(\ref{equation:semi_discrete_scheme})$. Then there exists $h_1 > 0$ such that
    \begin{align*}
      \max_{t \in [0,T]} \max_{i = 1,...,M} \Vert z_h(t) \Vert_{W^{3,\infty}(I_i)^d} \leq 2 c_0
    \end{align*}
    for all $h < h_1$, where $c_0 := \max_{t \in [0,T]} \Vert z(t) \Vert_{W^{3,\infty}(I)^d}$.
  \end{lemma}
  \begin{proof}
    Let $\varepsilon := \sup \{t \in [0,T]\ \vert \ \max_{i =1,...,M} \Vert z_h(s) \Vert_{W^{3,\infty}(I_i)^d} \leq 2 c_0 \text{ for all } 0 \leq s \leq t \}$.
    Since
    \begin{align*}
      \Vert z_h(0) \Vert_{W^{3,\infty}(I_i)^d}
      \leq \Vert z(0) \Vert_{W^{3,\infty}(I)^d} + \Vert z(0) - \J_{h,3}z(0) \Vert_{W^{3,\infty}(I_i)^d}
      \leq c_0 + c h^\frac{1}{2}
    \end{align*}
    we have that $\varepsilon > 0$. Assume that $\varepsilon < T$. Then we have for $i \in \{1,...,M\}$ and $t \in [0,\varepsilon]$
    \begin{align*}
      \Vert z_h(t) \Vert_{W^{3,\infty}(I_i)^d}
      &\leq \Vert z(t) \Vert_{W^{3,\infty}(I_i)^d}
      + \Vert z(t) - \I_{h,3} z(t) \Vert_{W^{3,\infty}(I_i)^d}
      + \Vert \I_{h,3}z(t) - z_h(t) \Vert_{W^{3,\infty}(I_i)^d} \\
      &\leq c_0 + c h^\frac{1}{2} \Vert z(t) \Vert_{H^4(I)^d}
      + c h^{-\frac{3}{2}} \Vert \I_{h,3} z(t) - z_h(t) \Vert_{H^2(I)^d}.
    \end{align*}
    Since $\max_{i = 1,...,M} \Vert z_h(t) \Vert_{W^{3,\infty}(I_i)^d} \leq 2 c_0$ for $t \in [0,\varepsilon]$ we can use Theorem \ref{theorem:error_estimate} on $[0,\varepsilon]$ and the interpolation estimate of Lemma \ref{lemma:interpolation_estimate} to obtain
    \begin{align*}
      \Vert \I_{h,3} z(t) - z_h(t) \Vert_{H^2(I)^d}
      \leq \Vert \I_{h,3} z(t) - z(t) \Vert_{H^2(I)^d}
      + \Vert z(t) - z_h(t) \Vert_{H^2(I)^d}
      \leq c h^2.
    \end{align*}
    Inserting this estimate into the previous estimate yields
    \begin{align*}
      \Vert z_h(t) \Vert_{W^{3,\infty}(I_i)^d}
      \leq c_0 + c h^\frac{1}{2}
      \leq \frac{3}{2} c_0
    \end{align*}
    for all $t \in [0,\varepsilon]$ provided that $0 < h \leq h_1$. Then there exists $\widetilde \varepsilon > \varepsilon$ such that
    \begin{align*}
      \max_{i =1,...,M} \Vert z_h(t) \Vert_{W^{3,\infty}(I_i)^d} \leq 2 c_0
    \end{align*}
    for all $t \in [0,\widetilde \varepsilon]$, contradicting the definition of $\varepsilon$.
  \end{proof}
  
  We have established convergence of the semi-discrete solutions and a quasi-optimal error estimate under suitable regularity assumptions. It remains to establish existence and approximability of semi-discrete solutions with a fully discrete scheme. This discrete scheme will be introduced in Section \ref{section:time_discretization} and existence of solutions and convergence of the scheme will also be proved there. We will follow a standard approach using an energy estimate to obtain a weakly convergent subsequence. For the full sequence of discrete solutions to converge, it is therefore necessary, that the semi-discrete solutions are unique.
  
  \begin{proposition}[Uniqueness of semi-discrete solutions]
    \label{theorem:semi_discrete_uniqueness}
    Solutions $z_h \in C^1([0,T]; \S^{3,1}(\T_h)^d)$ to the semi-discrete problem
    $(\ref{equation:semi_discrete_scheme})$ are
    unique.
  \end{proposition}
  \begin{proof}
    Let $z_h, \widetilde{z}_h$ be two solutions to the semi-discrete scheme
    (\ref{equation:semi_discrete_scheme}).
    We set
    \begin{align*}
      y_h(x) := (z_h - \widetilde{z}_h)(x) 
      - \int_a^x \I_{h,2}(((z_{hx} - \widetilde{z}_{hx})\cdot z_{hx}) z_{hx}) \ ds, \\
      \widetilde{y}_h(x) := (\widetilde{z}_h - z_h)(x)
      - \int_a^x \I_{h,2}(((\widetilde{z}_{hx} - z_{hx}) \cdot \widetilde{z}_{hx})
      \widetilde{z}_{hx}) \ ds.
    \end{align*}
    Therefore we have $y_h(t) \in \G_h(z_h(t))$, 
    $\widetilde{y}_h(t) \in \G_h(\widetilde{z}_h(t))$ for all $t$. Testing
    with these functions yields
    \begin{align*}
      0 &= \int_I z_{ht} \cdot y_h + z_{hxx} \cdot y_{hxx} \ dx \\
      &= \int_I z_{ht} \cdot (z_h - \widetilde{z}_h) 
      + z_{hxx} \cdot (z_h - \widetilde{z}_h)_{xx} \ dx \\
      &- \int_I z_{ht} \cdot \int_a^x \I_{h,2}(((z_{hx} - \widetilde{z}_{hx})\cdot z_{hx})
      z_{hx}) \ ds + z_{hxx} \cdot \I_{h,2}(((z_{hx} - \widetilde{z}_{hx})\cdot
      z_{hx}) z_{hx})_x \ dx,
    \end{align*}
    and analogously
    \begin{align*}
      0 &= \int_I \widetilde{z}_{ht} \cdot (\widetilde{z}_h - z_h) 
      + \widetilde{z}_{hxx} \cdot (\widetilde{z}_h - z_h)_{xx} \ dx \\
      &- \int_I \widetilde{z}_{ht} \cdot \int_a^x 
      \I_{h,2}(((\widetilde{z}_{hx} - z_{hx})\cdot \widetilde{z}_{hx})
      \widetilde{z}_{hx}) \ ds 
      + \widetilde{z}_{hxx} \cdot \I_{h,2}(((\widetilde{z}_{hx} - z_{hx})\cdot
      \widetilde{z}_{hx}) \widetilde{z}_{hx})_x \ dx.
    \end{align*}
    We add both equations and get
    \begin{align*}
      \int_I &(z_{ht} - \widetilde{z}_{ht}) \cdot (z_h - \widetilde{z}_h)
      + \vert z_{hxx} - \widetilde{z}_{hxx} \vert^2 \ dx \\
      &= \int_I z_{ht} \cdot \int_a^x \I_{h,2}(((z_{hx} - \widetilde{z}_{hx})\cdot z_{hx})
      z_{hx}) \ ds + z_{hxx} \cdot \I_{h,2}(((z_{hx} - \widetilde{z}_{hx})\cdot
      z_{hx}) z_{hx})_x \ dx \\
      &+ \int_I \widetilde{z}_{ht} \cdot \int_a^x 
      \I_{h,2}(((\widetilde{z}_{hx} - z_{hx})\cdot \widetilde{z}_{hx})
      \widetilde{z}_{hx}) \ ds 
      + \widetilde{z}_{hxx} \cdot \I_{h,2}(((\widetilde{z}_{hx} - z_{hx})\cdot
      \widetilde{z}_{hx}) \widetilde{z}_{hx})_x \ dx  \\
      &= \Rom{1} + \Rom{2} + \Rom{3} + \Rom{4}.
    \end{align*}
    For $\Rom{1}$ we obtain with Hölder's inequality and basic integral estimates
    \begin{align*}
      \Rom{1} 
      &= \int_I z_{ht} \cdot \int_a^x \I_{h,2}(((z_{hx} - \widetilde{z}_{hx})\cdot
      z_{hx})z_{hx}) \ ds \ dx \\
      &\leq \Vert z_{ht} \Vert_{L^1(I)^d} \left\Vert \int_a^x \I_{h,2}
      (((z_{hx} - \widetilde{z}_{hx})\cdot z_{hx})z_{hx}) \ ds \right\Vert_{L^\infty(I)^d} \\
      &\leq \Vert z_{ht} \Vert_{L^1(I)^d} \Vert \I_{h,2}
      (((z_{hx} - \widetilde{z}_{hx})\cdot z_{hx})z_{hx}) \Vert_{L^1(I)^d}.
    \end{align*}
    For $\Rom{2}$ we use additional inverse estimates from Lemma
    \ref{lemma:inverse_estimates} to obtain
    \begin{align*}
      \Rom{2}
      &= \int_I z_{hxx} \cdot \I_{h,2}(((z_{hx} - \widetilde{z}_{hx})\cdot
      z_{hx}) z_{hx})_x \ dx \\
      &\leq \Vert z_{hxx} \Vert_{L^\infty(I)^d}
      \Vert \I_{h,2}(((z_{hx} - \widetilde{z}_{hx})\cdot z_{hx}) z_{hx})_x \Vert_{L^1(I)^d}\\
      &\leq c h^{-\frac{3}{2}} \Vert z_{hxx} \Vert \Vert \I_{h,2} 
      (((z_{hx} - \widetilde{z}_{hx})\cdot z_{hx}) z_{hx}) \Vert_{L^1(I)^d}.
    \end{align*}
    Analogous estimates hold for $\Rom{3}$ and $\Rom{4}$.
    We can now use Lemma \ref{lemma:L1_estimate} and Lemma \ref{lemma:L1_estimate_2}
    to get
    \begin{align*}
      &\Vert \I_{h,2}
      (((z_{hx} - \widetilde{z}_{hx})\cdot z_{hx})z_{hx}) \Vert_{L^1(I)^d}
      \leq c \Vert \I_{h,2}((z_{hx} - \widetilde{z}_{hx})\cdot z_{hx}) \Vert_{L^1(I)} \\
      &\kurz 
      = \frac{c}{2} \Vert \I_{h,2} (\vert z_{hx} - \widetilde{z}_{hx} \vert^2) \Vert_{L^1(I)}
      \leq c \Vert z_{hx} - \widetilde{z}_{hx} \Vert_{L^2(I)^d}^2.
    \end{align*}
    With another inverse estimate and the energy estimate $\Vert z_{hxx} \Vert \leq
    c \Vert z_0 \Vert_{H^2(I)^d}$ we therefore get
    \begin{align*}
      \frac{1}{2} \frac{d}{dt} \Vert z_h - \widetilde{z}_h \Vert^2 
      + \Vert z_{hxx} - \widetilde{z}_{hxx} \Vert^2
      \leq C h^{-\frac{7}{2}}(1 + \Vert z_{ht} \Vert_{L^1(I)^d} + \Vert \widetilde{z}_{ht}
      \Vert_{L^1(I)^d}) \Vert z_h - \widetilde{z}_h \Vert^2.
    \end{align*}
    Through integration and application of Grönwall's inequality we obtain
    $z_h = \widetilde z_h$.
  \end{proof}
  
  % Time discretization
  \section{Time discretization}\label{section:time_discretization}
  In this section we construct a fully discrete scheme to approximate the semi-discrete problem (\ref{equation:semi_discrete_scheme}), similarly to the one from \cite{Bar13}, but adapted to the $\P_2$ constraint.
  For this we first dissect the time interval 
  $[0,T] = \bigcup_{n = 1}^N [t_{n-1}, t_n]$
  with $t_n = n\tau$ and time step size $\tau$. Let $Z^n \in \S^{3,1}(\T_h)^d$ the
  calculated approximation of $z_h(t_n)$. Note that the discrete constraint
  $\I_{h,2}(\vert z_{hx} \vert^2 - 1)= 0$ for the semi-discrete scheme can be
  imposed equivalently via the two equations
  \begin{align*}
    0 = \I_{h,2}(\vert z_{hx}(0) \vert^2 - 1), \leer
    0 = \frac{1}{2}\frac{d}{dt} \I_{h,2}(\vert z_{hx} \vert^2 - 1)
    = \I_{h,2}( z_{htx} \cdot z_{hx}).
  \end{align*}
  We now linearize this constraint with respect to the previous time step by
  replacing the time derivative in $z_{htx}$ with the backwards difference quotient
  $d_t Z^{n+1}_x$.
  We obtain the linearized discrete constraint
  \begin{align*}
    0 = \I_{h,2}(\vert Z^0_x \vert^2 - 1), \leer
    0 = \I_{h,2}(d_t^+ Z^n_x \cdot Z^n_x)
    =\I_{h,2}(d_tZ^{n+1}_x \cdot Z^n_x)
  \end{align*}
  for all $n \in \{0,...,N-1\}$.
  By also replacing the time derivative in the semi-discrete scheme with the 
  backwards difference quotient we obtain the fully discrete scheme:\\
  Set 
  $$
  Z^0 := \J_{h,3} z_0 = z_0(a) + \int_a^x I_{h,2}(z_0') \ d\sigma.
  $$ 
  Given $Z^n \in \S^{3,1}(\T_h)^d$ find $d_t Z^{n+1} \in \G_h(Z^n)$ such that
  \begin{equation}\label{equation:discrete_scheme}
    (d_t Z^{n+1}, Y) + \tau(d_t Z^{n+1}_{xx}, Y_{xx}) = - (Z^n_{xx}, Y_{xx})
  \end{equation}
  for all $Y \in \G_h(Z^n)$
  and set $Z^{n+1} = Z^n + \tau d_t Z^{n+1}$. Since the discretized, linearized
  constraint defines a closed subspace of $\S^{3,1}(\T_h)^d$, the existence of 
  discrete solutions follows immediately from the Lax-Milgram lemma.
  
  \subsection{Convergence of discrete solutions}
  Now that we have established the existence of discrete solutions $(Z^n)_{n=0,...,N}$ we interpolate those values to obtain functions that are defined on the entire time interval $[0,T]$. For this we define $\hat{Z}, Z^+, Z^-: [0,T] \to \S^{3,1}(\T_h)^d$ via
  $$
  \hat{Z}(0) = Z^+(0) = Z^-(0) = Z^0
  $$ 
  and
  \begin{equation}
    \label{equation:time interpolation}
    \hat{Z}(t) := Z^n + (t - t_n) d_t Z^{n+1},
    \leer Z^+(t) := Z^{n+1},
    \leer Z^-(t) := Z^n
  \end{equation}
  for $t \in (t_n, t_{n+1}]$.
  Now we want to show that these interpolants converge as $\tau \to 0$ and that their limit function is a solution to the semi-discrete problem (\ref{equation:semi_discrete_scheme}). For the convergence of those functions we need an a priori estimate that bounds them in $H^1(0,T;\S^{3,1}(\T_h)^d)$ and thus allows us to pick a weakly convergent subsequence. For the weak limit to be a possible solution to the semi-discrete problem, we also have to make sure it satisfies the discrete arc-length constraint. For this we have to control the discrete constraint violation of the interpolants and show that it vanishes in the limit as $\tau \to 0$.
  
  \begin{proposition}[discrete energy stability]
    \label{proposition:discrete_energy_stabiltiy}
    The discrete solutions satisfy for all $n \in \mathbb{N}$
    \begin{equation}\label{equation:discrete_energy_stabiltiy}
      \frac{1}{2} \Vert Z^n_{xx} \Vert^2
      + \tau \sum_{k=0}^{n-1} \left( \Vert d_t Z^{k+1} \Vert^2 
      + \frac{\tau}{2} \Vert d_t Z^{k+1}_{xx} \Vert^2 \right)
      = \frac{1}{2} \Vert Z^0_{xx} \Vert^2.
    \end{equation}
    This especially implies $\Vert Z^n_{xx} \Vert \leq \Vert Z^{n-1}_{xx} \Vert
    \leq ... \leq \Vert Z^0_{xx} \Vert$.
  \end{proposition}
  
  \begin{proof}
    Testing the discrete scheme (\ref{equation:discrete_scheme}) with $d_t Z^{k+1}
    \in \G_h(Z^k)$ yields
    \begin{align*}
      0 = \Vert d_t Z^{k+1} \Vert^2 + \tau \Vert d_t Z^{k+1}_{xx} \Vert^2
      + (Z^k_{xx}, d_t Z^{k+1}_{xx}).
    \end{align*}
    From the identity $Z^{k+1} = Z^k + \tau d_t Z^{k+1}$ and the binomial formula we have
    \begin{align*}
      \Vert Z^{k+1}_{xx} \Vert^2
      = \Vert Z^k_{xx} \Vert^2 + 2\tau (Z^k_{xx}, d_t Z^{k+1}_{xx})
      + \tau^2 \Vert d_t Z^{k+1}_{xx} \Vert^2,
    \end{align*}
    which is equivalent to
    \begin{align*}
      (Z^k_{xx}, d_t Z^{k+1}_{xx})
      = \frac{1}{2\tau}(\Vert Z^{k+1}_{xx} \Vert^2 - \Vert Z^k_{xx} \Vert^2)
      - \frac{\tau}{2} \Vert d_t Z^{k+1}_{xx} \Vert^2
      = \frac{d_t}{2} \Vert Z^{k+1}_{xx} \Vert^2 
      - \frac{\tau}{2} \Vert d_t Z^{k+1}_{xx} \Vert^2.
    \end{align*}
    Inserting this identity into the first equation yields
    \begin{align*}
      0 = \frac{d_t}{2} \Vert Z^{k+1}_{xx} \Vert^2 + \Vert d_t Z^{k+1} \Vert^2
      + \frac{\tau}{2} \Vert d_t Z^{k+1}_{xx} \Vert^2.
    \end{align*}
    Multiplying both sides with $\tau$ and summation over $k = 0,...,n-1$ then 
    implies
    \begin{align*}
      0 &= \tau \sum_{k=0}^{n-1} \frac{d_t}{2} \Vert Z^{k+1}_{xx} \Vert^2
      + \Vert d_t Z^{k+1} \Vert^2 + \frac{\tau}{2} \Vert d_t Z^{k+1}_{xx} 
      \Vert^2 \\
      &= \frac{1}{2}( \Vert Z^n_{xx} \Vert^2 - \Vert Z^0_{xx} \Vert^2)
      + \sum_{k=0}^{n-1} \tau \Vert d_t Z^{k+1} \Vert^2
      + \frac{\tau^2}{2} \Vert d_t Z^{k+1}_{xx} \Vert^2,
    \end{align*}
    which is equivalent to the asserted equality.
  \end{proof}
  
  \begin{proposition}[discrete constraint violation]
    \label{proposition:discrete_constraint_violation}
    For all $t \in [0,T]$ and $\widetilde x \in \N_2(\T_h)$ we have
    \begin{align*}
      \vert \vert \hat Z_x(\widetilde{x},t) \vert^2 - 1 \vert
      \leq c \Vert Z^0_{xx} \Vert^2 \tau^\frac{1}{2} h^{-1} 
      (1 + \tau^\frac{1}{2}).
    \end{align*}
  \end{proposition}
  
  \begin{proof}
    Let $t \in (t_n,t_{n+1}]$, $\widetilde{x} \in \N_2(\T_h)$. Since $Z^{k+1}_x = Z^k_x + \tau
    d_t Z^{k+1}_x$ and $Z^k_x(\widetilde x) \cdot d_t Z^{k+1}_x(\widetilde x) = 0$, we have
    \begin{align*}
      \vert Z^{k+1}_x(\widetilde{x}) \vert^2
      = \vert Z^k_x(\widetilde{x}) \vert^2
      + \tau^2 \vert d_t Z^{k+1}_x(\widetilde{x}) \vert^2
      + 2\tau Z^k_x(\widetilde{x}) \cdot d_t Z^{k+1}_x(\widetilde{x})
      = \vert Z^k_x(\widetilde{x}) \vert^2
      + \tau^2 \vert d_t Z^{k+1}_x(\widetilde{x}) \vert^2.
    \end{align*}
    Therefore we have for all $n$
    \begin{align*}
      \vert Z^n_x(\widetilde{x}) \vert^2
      = \vert Z^0_x(\widetilde{x}) \vert^2
      + \sum_{k=0}^{n-1} \tau^2 \vert d_t Z^{k+1}_x(\widetilde{x}) \vert^2
    \end{align*}
    and similarly
    \begin{align*}
      \vert \hat Z_x(\widetilde x,t) \vert^2
      = \vert Z^n_x(\widetilde x) \vert^2 + (t - t_n)^2 \vert d_t Z^{n+1}(\widetilde x) \vert^2.
    \end{align*}
    Since $\vert Z^0_x(\widetilde x) \vert^2 = 1$, with the help of the inverse estimate from Lemma \ref{lemma:inverse_estimates} we obtain
    \begin{align*}
      \vert \vert \hat Z_x(\widetilde{x},t) \vert^2 - 1 \vert
      \leq \tau^2 \sum_{k=0}^{n} \vert d_t Z^{k+1}_x(\widetilde{x}) \vert^2
      \leq c \tau^2 h^{-1} \sum_{k=0}^{n} \Vert d_t Z^{k+1}_x \Vert^2.
    \end{align*}
    The Gagliardo-Nirenberg inequality from Lemma \ref{lemma:gagliardo_nirenberg}
    thus implies
    \begin{align*}
      \vert \vert \hat Z_x(\widetilde{x},t) \vert^2 - 1 \vert
      &\leq c \tau^\frac{1}{2} h^{-1} \sum_{k=0}^{n} \tau^\frac{1}{2} 
      \Vert d_t Z^{k+1} \Vert \tau
      \Vert d_t Z^{k+1}_{xx} \Vert + \tau^\frac{3}{2} \Vert d_t Z^{k+1} \Vert^2 \\
      &\leq c \tau^\frac{1}{2} h^{-1} \sum_{k=0}^{n} \tau 
      \Vert d_t Z^{k+1} \Vert^2 +  \tau^2
      \Vert d_t Z^{k+1}_{xx} \Vert^2 + \tau^\frac{3}{2} \Vert d_t Z^{k+1} \Vert^2.
    \end{align*}
    The energy estimate (\ref{equation:discrete_energy_stabiltiy}) then yields
    \begin{align*}
      \vert \vert \hat Z_x(\widetilde{x},t) \vert^2 - 1 \vert
      \leq c \Vert Z^0_{xx} \Vert^2 \tau^\frac{1}{2} h^{-1} 
      (1 + \tau^\frac{1}{2}),
    \end{align*}
    which proves the estimate.
  \end{proof}
  
  Next we will show, that those discrete solutions
  converge towards a solution $z_h$ of the semi-discrete problem as $\tau \to 0$.
  This will also prove the existence of semi-discrete solutions that we asserted
  in the proof of Theorem \ref{theorem:error_estimate}.
  
  \begin{proposition}[Convergence of the discrete scheme]
    Let $\{Z^n \ \vert \ n \in \{0,...,N\}\} \subset \S^{3,1}(\T_h)^d$ be the calculated discrete
    solutions and $\hat{Z},\ Z^+$ and $Z^-$ the interpolants defined in $(\ref{equation:time interpolation})$.
    Then there is $z_h \in H^1([0,T],\S^{3,1}(\T_h)^d)$ such that $\hat{Z} \rightharpoonup z_h$ in
    $H^1([0,T],\S^{3,1}(\T_h)^d)$ as $\tau \to 0$. Further $z_h$ satisfies the discrete constraint
    $\I_{h,2}(\vert z_{hx} \vert^2 - 1) = 0$ and is the unique solution to the 
    semi-discrete scheme.
  \end{proposition}
  \begin{proof}
    From Proposition \ref{proposition:discrete_energy_stabiltiy} we know that
    \begin{align*}
      \frac{1}{2} \Vert Z^n_{xx} \Vert^2
      + \sum_{k=0}^{n-1} \tau \Vert d_t Z^{k+1} \Vert^2 + \frac{\tau^2}{2}
      \Vert d_t Z^{k+1}_{xx} \Vert^2
      = \frac{1}{2} \Vert Z^0_{xx} \Vert^2.
    \end{align*}
    By definition we have $\del_t \hat{Z} \vert_{(t_k, t_{k+1})} = 
    d_t Z^{k+1}$ and thus
    \begin{align*}
      \sum_{k=0}^{n-1} \tau \Vert d_t Z^{k+1} \Vert^2 + \frac{\tau^2}{2}
      \Vert d_t Z^{k+1}_{xx} \Vert^2
      = \int_0^{t_n} \Vert \del_t \hat{Z} \Vert^2 
      + \frac{\tau}{2} \Vert \del_t \hat{Z}_{xx} \Vert^2 \ dr.
    \end{align*}
    Since for $t \in (t_{k-1}, t_k]$ we have $Z^+(t) = Z^k$ and the integral
    term increases monotonically, we get
    \begin{equation*}
      \frac{1}{2} \Vert Z^+_{xx}(t) \Vert^2
      + \int_0^t \Vert \del_t \hat{Z} \Vert^2 + \frac{\tau}{2} \Vert \del_t \hat{Z}_{xx}
      \Vert^2 \ ds
      \leq \frac{1}{2} \Vert Z^0_{xx} \Vert^2
    \end{equation*}
    for all $t \in [0,T]$.
    By definition we have $\vert \hat{Z} - Z^\pm \vert \leq \tau \vert \del_t \hat{Z} \vert$
    and thus $\hat{Z}$ is bounded in $H^1([0,T]; \S^{3,1}(\T_h)^d)$ and since $\S^{3,1}(\T_h)^d$ has
    finite dimension, we can extract a subsequence such that
    $\hat{Z} \rightharpoonup z_h$ in $H^1([0,T];\S^{3,1}(\T_h)^d)$ as $\tau \to 0$. The Sobolev
    embedding theorem therefore implies
    $\hat{Z} \to z_h$ in $C^0([0,T];\S^{3,1}(\T_h)^d)$. Further we have
    $Z^\pm \to z_h$ in $L^\infty([0,T];\S^{3,1}(\T_h)^d)$, since
    \begin{align*}
      \Vert \hat Z - Z^\pm \Vert_{L^\infty([0,T];H^2(I)^d)}^2
      &\leq \tau^2 \Vert \del_t \hat Z \Vert_{L^\infty([0,T];H^2(I)^d)}^2
      = \tau^2 \max_{n = 1,...,N} \Vert d_t Z^n \Vert_{H^2(I)^d}^2 \\
      &\leq c h^{-4} \tau^2 \sum_{n=1}^N \Vert d_t Z^n \Vert^2
      \leq c h^{-4} \tau \Vert Z^0_{xx} \Vert^2
      \xrightarrow{\tau \to 0} 0.
    \end{align*}
    Let now $\widetilde{x} \in \N_2(\T_h), t \in [0,T]$
    arbitrary. From Proposition \ref{proposition:discrete_constraint_violation}
    we get
    \begin{align*}
      \vert \vert z_{hx} (\widetilde x,t) \vert^2 - 1 \vert &= 
      \lim_{\tau \to 0} \vert \vert \hat Z_x(\widetilde{x},t)\vert^2 - 1 \vert
      \leq \lim_{\tau \to 0} c \tau^\frac{1}{2} h^{-1} = 0.
    \end{align*}
    Thus $z_h$ satisfies the constraint $\I_{h,2}(\vert z_{hx} \vert^2 - 1) = 0$.
    Let now $y_h \in \S^{3,1}_D(\T_h)^d$ be arbitrary. We set
    \begin{align*}
      Y^n(x) := y_h(x) - \int_a^x \I_{h,2} \left( (y_{hx} \cdot Z^n_x) \frac{Z^n_x}{\vert Z^n_x \vert^2} \right) \ d\sigma.
    \end{align*}
    This gives us
    \begin{align*}
      \I_{h,2}(Y^n_x \cdot Z^n_x)
      = \I_{h,2} \left( y_{hx} \cdot Z^n_x - (y_{hx} \cdot Z^n_x) \frac{\vert Z^n_x \vert^2}{\vert Z^n_x \vert^2} \right) = 0,
    \end{align*}
    thus $Y^n \in \G_h(Z^n)$. Testing the discrete scheme $(\ref{equation:discrete_scheme})$ with $Y^n$ yields
    \begin{align*}
      0 &= (d_t Z^{n+1}, Y^n) + (Z^{n+1}_{xx}, Y^n_{xx}) 
      = (d_t Z^{n+1}, y_h) + (Z^{n+1}_{xx}, y_{hxx}) \\
      &-\left(d_t Z^{n+1}, \int_a^x \I_{h,2} \left( (y_{hx} \cdot Z^n_x) \frac{Z^n_x}{\vert Z^n_x \vert^2}\right) \ d\sigma \right)
      - \left( Z^{n+1}_{xx}, \I_{h,2} \left( (y_{hx} \cdot Z^n_x) \frac{Z^n_x}{\vert Z^n_x \vert^2} \right)_x \right).
    \end{align*}
    We now multiply this equation with $\eta \in C^\infty_c((0,T))$ arbitrary and integrate it over $(t_n, t_{n+1})$. Using the definitions from $(\ref{equation:time interpolation})$ and the identity $\del_t \hat Z \vert_{(t_n,t_{n+1})} = d_t Z^{n+1}$ we can sum up over all $n = 0,...,N$ to get
    \begin{align*}
      &\int_0^T \eta(t) ((\del_t \hat Z(t), y_h) + (Z^+_{xx}(t), y_{hxx})) \ dt \\
      &\kurz = \int_0^T \eta(t) \left( \del_t \hat Z(t), \int_a^x \I_{h,2} \left((y_{hx} \cdot Z^-_x(t)) \frac{Z^-_x(t)}{\vert Z^-_x(t) \vert^2} \right) d\sigma \right) dt \\
      &\leer + \int_0^T \eta(t) \left(Z^+_{xx}(t), \I_{h,2} \left( (y_{hx} \cdot Z^-_x) \frac{Z^-_x}{\vert Z^-_x \vert^2}\right)_x \right) dt
    \end{align*}
    Passing to the limit $\tau \to 0$ and observing $\vert z_{hx}(\widetilde x, t) \vert^2 = 1,\ \widetilde x \in \N_2(\T_h)$ thus yields
    \begin{align*}
      &\int_0^T \eta(t) ((z_{ht}(t), y_h) + (z_{hxx}(t), y_{hxx})) \ dt \\
      &\kurz = \int_0^T \eta(t) \left( z_{ht}(t), \int_a^x \I_{h,2}((y_{hx} \cdot z_{hx}(t)) z_{hx}(t)) \ d\sigma \right) \ dt \\
      &\leer + \int_0^T \eta(t)(z_{hxx}(t), \I_{h,2}((y_{hx} \cdot z_{hx}(t)). z_{hx}(t))_x) \ dt
    \end{align*}
    And since $\eta \in C^\infty_c((0,T))$ and $y_h \in \S^{3,1}_D(\T_h)^d$ were chosen arbitrarily, the fundamental lemma in the calculus of variations implies
    \begin{align*}
      (z_{ht}, y_h) + (z_{hxx}, y_{hxx})
      = \left(z_{ht}, \int_a^x \I_{h,2}((y_{hx} \cdot z_{hx})z_{hx})\ d\sigma \right) + (z_{hxx}, \I_{h,2}((y_{hx} \cdot z_{hx}) z_{hx})_x)
    \end{align*}
    for every $y_h \in \S^{3,1}_D(\T_h)^d$ and almost everywhere on $(0,T)$. In particular we deduce for $y_h \in \G_h(z_h(t))$ that
    \begin{align*}
      (z_{ht}, y_h) + (z_{hxx}, y_{hxx}) = 0
    \end{align*}
    and thus $z_h$ solves the semi-discrete problem $(\ref{equation:semi_discrete_scheme})$.
  \end{proof}
  
  % Boundary conditions
  \subsection{Boundary conditions}
  We note that just like in \cite{Bar13} we can add fixed, clamped or
  periodic boundary conditions to the discrete scheme by choosing a starting value
  $Z^0$ for the iteration that satisfies the boundary conditions and
  enforce the additional conditions
  \begin{itemize}
    \item $d_t Z^{n+1} = 0 \text{ on } \Gamma_D $, $ d_t Z^{n+1}_x = 0
    \text{ on } 
    \Gamma_D' \kurz$ for fixed/clamped boundary conditions,
    \item $d_t Z^{n+1}(a) = d_t Z^{n+1}(b),\ 
    d_t Z^{n+1}_x(a) = d_t Z^{n+1}_x(b) \kurz $ for periodic boundary conditions.
  \end{itemize}
  In the case of $\Gamma_D = \del I$, this however introduces a new problem.
  Since $Z^0$ must also satisfy the discrete arc-length constraint 
  $\I_{h,2}(\vert Z^0_x \vert - 1) = 0$ and the nodal interpolant of $z_0$ in
  general does not satisfy this constraint, we set $Z^0 := \J_{h,3} z_0$.
  The problem with this approach is, that in general we have
  $\J_{h,3} z_0(b) \neq z_0(b)$. Thus our choice of $Z^0$ will lead to a discrete
  solution that does not satisfy the required boundary conditions. This however
  is not a big problem, as for the error term $z_0(b) - \J_{h,3}z_0(b)$ we have
  $\vert z_0(b) - \J_{h,3} z_0(b) \vert \leq c h^4$ according to
  Lemma \ref{lemma:interpolation_estimate}.
  Another option to impose clamped boundary conditions is to introduce a penalty term $\varepsilon^{-1} \vert (z_h - u_D)(b) \vert^2$ to the energy functional instead.
  
  % Numerical Simulations
  \section{Numerical Experiments}\label{section:numerical_experiments}
  In this section we perform numerical experiments to verify the approximation
  results from Theorem \ref{theorem:error_estimate}. To be able to properly calculate the approximation error $z - z_h$ we choose an initial value $z_0$ for which the continuous elastic flow $z$ is well known. This is for example the case if $z_0$ is stationary. To calculate the starting value
  \begin{align*}
    z_{h,0} = \J_{h,3} z_0 = z_0(a) + \int_a^x \I_{h,2}(z_0') \ d\sigma
  \end{align*}
  in the Hermite basis we use the explicit formula
  \begin{align*}
    \int_{x_i}^{x_{i+1}} \I_{h,2} f \ dx 
    = \frac{h_i}{6}(f(x_i) + 4 f(m_i) + f(x_{i+1}))
  \end{align*}
  for the Simpson rule to calculate
  \begin{gather*}
    z_{h,0}(x_0) = z_0(x_0),\leer
    z_{h,0}'(x_i) = z_0'(x_i) \kurz \text{ for all } i = 0,...,M, \\
    z_{h,0}(x_i) = z_{h,0}(x_{i-1}) + \frac{h_i}{6}(z_0'(x_{i-1}) + 4 z_0'(m_i) + z_0'(x_i)) \kurz \text{ for all } i = 1,...,M.
  \end{gather*}
  To compute the norms involved in the error estimate we set $e^n_h := z(t_n) - Z^n$, $e_{ht}^n := \I_{h,3} z_t(t_n) - d_t Z^n$ and use the approximations
  \begin{align*}
    \vert e_h \vert_{L^\infty H^2}
    := \max_n \vert e_h^n \vert_{H^2(I)^d}, \leer
    \vert e_h \vert_{H^1 L^2}^2
    := \tau \sum_{n=1}^N \Vert e_{ht}^n \Vert^2.
  \end{align*}
  Since $e_{ht}^n \in \S^{3,1}(\T_h)$ the term $\vert e_h \vert_{H^1L^2}$ can be computed exactly. For the computation of $\vert e_h^n \vert_{H^2(I)^d}$ we use the binomial identity to get
  \begin{align*}
    \vert e_h^n \vert_{H^2(I)^d}^2
    = \vert z(t_n) \vert_{H^2(I)^d}^2 + \vert Z^n \vert_{H^2(I)^d}^2
    - 2 \int_I z_{xx}(t_n) \cdot Z^n_{xx} \ dx.
  \end{align*}
  Using Lemma \ref{lemma:rhs_interpolation}, we can replace $z_{xx}(t_n)$ in the integral by the second derivative of its nodal interpolant $\I_{h,3}z(t_n)$ to obtain
  \begin{equation*}
    \vert e_h^n \vert_{H^2(I)^d}^2
    = \vert z(t_n) \vert_{H^2(I)^d}^2 + \vert Z^n \vert_{H^2(I)^d}^2
    - 2 \int_I (\I_{h,3}z(t_n))_{xx} \cdot Z^n_{xx} \ dx.
  \end{equation*}
  Additionally we also compute approximation for the approximation error $e_h$ in the weaker $L^\infty H^1$ semi-norm and the $L^\infty L^2$ norm by setting $\tilde e_h^n := \I_{h,3} z(t_n) - Z^n$ and
  \begin{align*}
    \vert \widetilde e_h \vert_{L^\infty H^1} := \max_n \vert \tilde e_h^n \vert_{H^1(I)^d}, \leer
    \Vert \widetilde e_h \Vert_{L^\infty L^2} := \max_n \Vert \tilde e_h^n \Vert.
  \end{align*}
  We start with a two-dimensional example.
  \begin{example}[Semi-clamped circle]
    \label{experiment:circle}
    We choose $I = [0, 2\pi ]$ and $z_0(x) := (\cos(x), \sin(x))$. 
    Additionally we choose $\Gamma_D = \{0\}$, $\Gamma_D' = \{0,2\pi\}$ and $T = 50$.
    Then $z_0$ is a local minimum for the bending energy and thus a solution to the elastic flow.
    Since $z(x,t) = z_0(x) = (\cos x, \sin x)$, we have $\vert z(t_n) \vert_{H^2(I)^d}^2
    = 2\pi$ and $z_t = 0$. 
    Now we calculate the approximation errors $\vert e_h \vert_{L^\infty H^2}$ and $\vert e_h \vert_{H^1 L^2}$ for both, the $\P_1$ and $\P_2$ constraint,
    as described above. The results are shown in Table \ref{table:maxH2 error circle} and \ref{table:H1L2 error circle}.
    The results for $\vert e_h \vert_{L^\infty H^2}$ are as expected. When it comes to $\vert e_h \vert_{H^1 L^2}$, we observe that the semi-discrete flow is constant in case of the $\P_2$ constraint. That means if $z$ is a local minimizer of the bending energy $E$ under the continuous arc-length constraint, then $\J_{h,3} z$ minimizes $E$ locally under the $\P_2$ constraint. 
    This also implies that in the case of the $\P_2$ constraint the approximation error is the same as the interpolation error and is therefore quasi-optimal in the weaker norms as well. In contrast for the $\P_1$ constraint we observe suboptimal quadratic convergence in both weaker norms.
    \\
    \begin{filecontents*}{tables/circle_open_end_h2.csv}
1.57080,1.290e+00,-,1.340e+00,-,1.371e+00,-,2.228e-01,-,2.228e-01,-,2.228e-01,-
0.78540,5.628e-01,1.19662,5.629e-01,1.25179,5.629e-01,1.28421,5.714e-02,1.96322,5.714e-02,1.96322,5.714e-02,1.96322
0.39270,2.834e-01,0.98972,2.834e-01,0.98979,2.834e-01,0.98984,1.438e-02,1.99081,1.438e-02,1.99081,1.438e-02,1.99081
0.19635,1.420e-01,0.99724,1.420e-01,0.99726,1.420e-01,0.99727,3.600e-03,1.99770,3.600e-03,1.99770,3.600e-03,1.99770
0.09817,7.103e-02,0.99930,7.103e-02,0.99930,7.103e-02,0.99931,9.003e-04,1.99939,9.003e-04,1.99939,9.003e-04,1.99939
\end{filecontents*}
\DTLloaddb[noheader,keys={h,e11,r11,e12,r12,e13,r13,e21,r21,e22,r22,e23,r23}]
    {circle_h2_DB}
    {tables/circle_open_end_h2.csv}   
    \begin{center}\begin{minipage}{.85\linewidth}
        \resizebox{\linewidth}{!}{
          \begin{tabular}{c|cc|cc|cc|cc}\toprule
            \multirow{3}{*}{$h$}
            & \multicolumn{4}{c|}{$\mathcal{P}_1$ constraint}
            & \multicolumn{4}{c}{$\mathcal{P}_2$ constraint} \\
            \cmidrule{2-9}
            & \multicolumn{2}{c|}{$\tau = 1/10$} 
            & \multicolumn{2}{c|}{$\tau = 1/20$}
            & \multicolumn{2}{c|}{$\tau = 1/10$}
            & \multicolumn{2}{c}{$\tau = 1/20$}\\
            & $\vert e_{h} \vert_{L^\infty H^2}$ & eoc 
            & $\vert e_{h} \vert_{L^\infty H^2}$ & eoc
            & $\vert e_{h} \vert_{L^\infty H^2}$ & eoc
            & $\vert e_{h} \vert_{L^\infty H^2}$ & eoc
            \DTLforeach*{circle_h2_DB}{\h=h,\a=e11,\b=r11,\c=e12,\d=r12,\e=e21,\f=r21,
              \g=e22,\i=r22}{%
              \DTLiffirstrow{\\\cmidrule{1-9}}{\\}%
              \h & \a & \b & \c & \d & \e & \f & \g & \i
            }%
            \\\bottomrule
        \end{tabular}}
        \captionof{table}
        {
          \label{table:maxH2 error circle}
          Approximation error in Example \ref{experiment:circle} for the schemes with $\P_1$ and $\P_2$ 
          constraint in the $L^\infty H^2$ semi-norm for various step and mesh sizes. 
          In case of the $\P_1$ constraint we can observe a linear convergence rate 
          as $h \to 0$, while for the $\P_2$ constraint we observe quadratic 
          convergence.
        }
    \end{minipage}\end{center}
    \begin{filecontents*}{tables/circle_open_end_time.csv}
1.57080,7.775e-01,-,7.782e-01,-,1.499e-14,-,1.264e-14,-
0.78540,3.937e-01,0.98167,3.952e-01,0.97758,1.989e-12,-7.05255,4.072e-12,-8.33204
0.39270,1.941e-01,1.02053,1.968e-01,1.00553,1.040e-12,0.93601,2.201e-12,0.88758
0.19635,9.087e-02,1.09486,9.476e-02,1.05472,5.684e-12,-2.45066,6.362e-12,-1.53117
0.09817,3.774e-02,1.26785,4.229e-02,1.16387,6.356e-11,-3.48294,7.040e-11,-3.46808
\end{filecontents*}
\DTLloaddb[noheader,keys={h,e11,r11,e12,r12,e21,r21,e22,r22}]
    {circle_time_DB}
    {tables/circle_open_end_time.csv}  
    \begin{center}\begin{minipage}{.85\linewidth}
        \resizebox{\linewidth}{!}{
          \begin{tabular}{c|cc|cc|cc|cc}\toprule
            \multirow{3}{*}{$h$}
            & \multicolumn{4}{c|}{$\mathcal{P}_1$ constraint}
            & \multicolumn{4}{c}{$\mathcal{P}_2$ constraint} \\
            \cmidrule{2-9}
            & \multicolumn{2}{c|}{$\tau = 1/1000$} 
            & \multicolumn{2}{c|}{$\tau = 1/2000$}
            & \multicolumn{2}{c|}{$\tau = 1/1000$}
            & \multicolumn{2}{c}{$\tau = 1/2000$}\\
            & $\vert e_{h} \vert_{H^1 L^2}$ & eoc 
            & $\vert e_{h} \vert_{H^1 L^2}$ & eoc
            & $\vert e_{h} \vert_{H^1 L^2}$ & eoc
            & $\vert e_{h} \vert_{H^1 L^2}$ & eoc
            \DTLforeach*{circle_time_DB}
            {\h=h,\a=e11,\b=r11,\c=e12,\d=r12,\e=e21,\f=r21,\g=e22,\i=r22}{%
              \DTLiffirstrow{\\\cmidrule{1-9}}{\\}%
              \h & \a & \b & \c & \d & \e & \f & \g & \i
            }%
            \\\bottomrule
        \end{tabular}}
        \captionof{table}
        {
          \label{table:H1L2 error circle}
          Approximation error in Example \ref{experiment:circle} for the $\P_1$ and $\P_2$ constraints in $H^1L^2$. For the $\P_1$ constraint we observe linear convergence. For the scheme with $\P_2$ constraint however, the discrete solution is stationary, just like the continuous one.
        }    
    \end{minipage}\end{center}
    
    \begin{filecontents*}{tables/circle_p1_weak_norms_combined.csv}
1.57080,5.801e-01,-,5.799e-01,-,5.562e-01,-,5.575e-01,-
0.78540,1.773e-01,1.70998,1.774e-01,1.70917,1.409e-01,1.98091,1.410e-01,1.98351
0.39270,4.506e-02,1.97622,4.507e-02,1.97637,3.558e-02,1.98565,3.560e-02,1.98567
0.19635,1.130e-02,1.99521,1.131e-02,1.99525,8.918e-03,1.99623,8.922e-03,1.99623
0.09817,2.828e-03,1.99885,2.829e-03,1.99886,2.231e-03,1.99905,2.232e-03,1.99905
\end{filecontents*}
\DTLloaddb[noheader,keys={h,e11,r11,e12,r12,e21,r21,e22,r22}]
    {circle_p1_DB}
    {tables/circle_p1_weak_norms_combined.csv}   
    \begin{center}\begin{minipage}{.85\linewidth}
        \resizebox{\linewidth}{!}{
          \begin{tabular}{c|cc|cc|cc|cc}\toprule
            \multirow{3}{*}{$h$}
            & \multicolumn{8}{c}{$\mathcal{P}_1$ constraint} \\
            \cmidrule{2-9}
            & \multicolumn{2}{c|}{$\tau = 1/10$}
            & \multicolumn{2}{c|}{$\tau = 1/20$}
            & \multicolumn{2}{c|}{$\tau = 1/10$}
            & \multicolumn{2}{c}{$\tau = 1/20$}\\
            & $\Vert \widetilde e_{h} \Vert_{L^\infty L^2}$ & eoc
            & $\Vert \widetilde e_{h} \Vert_{L^\infty L^2}$ & eoc
            & $\vert \widetilde e_{h} \vert_{L^\infty H^1}$ & eoc
            & $\vert \widetilde e_{h} \vert_{L^\infty H^1}$ & eoc
            \DTLforeach*{circle_p1_DB}{\h=h,\a=e11,\b=r11,\c=e12,\d=r12,
              \e=e21,\f=r21,\g=e22,\i=r22}{%
              \DTLiffirstrow{\\\cmidrule{1-9}}{\\}%
              \h & \a & \b & \c & \d & \e & \f & \g & \i
            }%
            \\\bottomrule
        \end{tabular}}
        \captionof{table}
        {
          \label{table:p1 errors circle}
          Approximation error in Example \ref{experiment:circle} in the $L^\infty L^2$ norm and the $L^\infty H^1$ semi-norm for the scheme using the $\P_1$ constraint. In both cases we observe quadratic convergence as $h \to 0$.
        }
    \end{minipage}\end{center}
  \end{example}
  
  \begin{figure}
    \begin{center}
      \begin{minipage}{0.3\linewidth}
        \includegraphics[trim={4cm 15mm 5cm 15mm},clip,width=\textwidth]
        {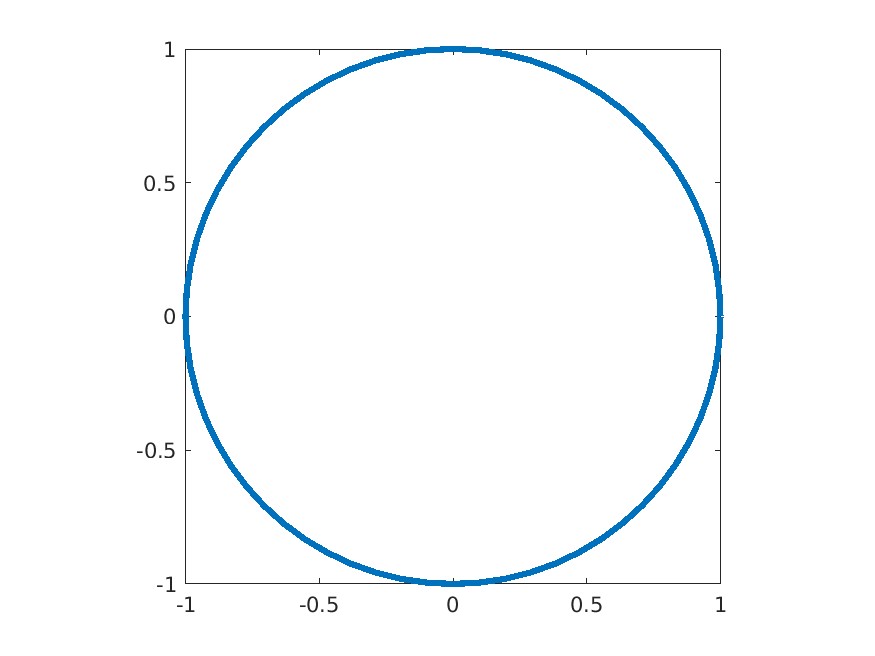}
      \end{minipage}
      \begin{minipage}{0.3\linewidth}
        \includegraphics[trim={4cm 0cm 4cm 2cm},clip,width=\textwidth]
        {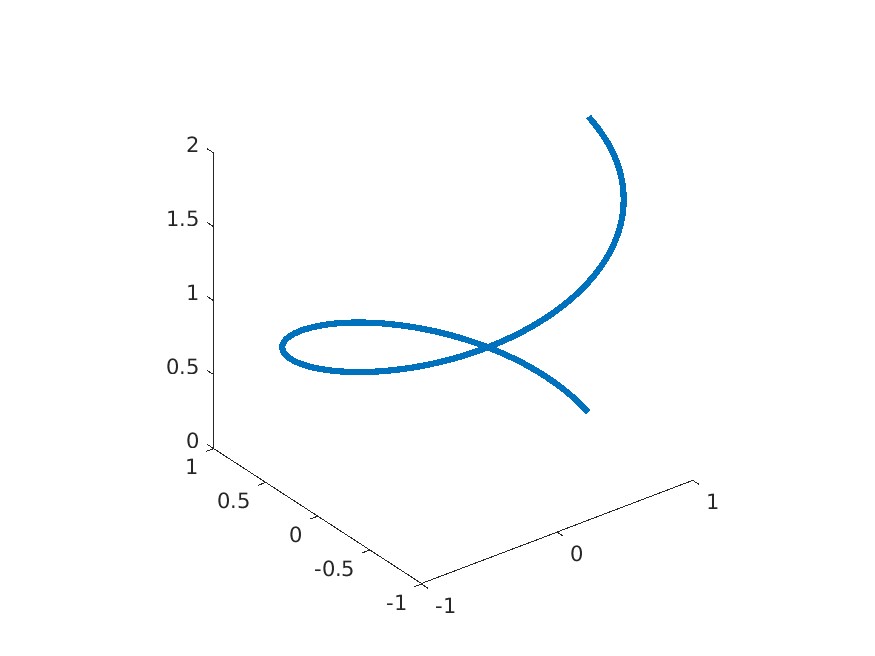}
      \end{minipage}
    \end{center}
    \caption{ \label{figure:circle and helix}
      Initial values $z_{h,0}$ for Example \ref{experiment:circle} (left)
      and Example \ref{experiment:helix} (right)}
  \end{figure}
  
  \begin{example}[Clamped helix]\label{experiment:helix}
    Now we give an example in three-dimensional space.\\
    Choose $I = [0, 2\sqrt{\pi^2 + 1}]$, $\lambda = \pi / \sqrt{\pi^2 + 1}$,
    $\mu = 1 / \sqrt{\pi^2 + 1}$, $T = 50$ and define $z_0: I \to \real^3$ via
    $$
    z_0(x) := (\cos(\lambda x), \sin(\lambda x), \mu x).
    $$ 
    This curve describes a helix as depicted in Figure \ref{figure:circle and helix} and for clamped boundary conditions, i.e. $\Gamma_D = \Gamma_D' = \del I$, $z_0$ is minimal for the bending energy and thus a solution to the elastic flow.
    We again calculate the approximation errors for the $\P_1$ and $\P_2$ discretization of the arc-length constraint as described above. The results are shown in Table \ref{table:maxH2 error helix} and Table \ref{table:H1L2 error helix}.
    For $\vert e_h \vert_{L^\infty H^2}$ we observe pretty much the same results as for the circle which is interesting, because it means that the results of Theorem \ref{theorem:error_estimate} also apply for clamped boundary conditions, even though this case is not covered by our proof. When it comes to the time derivative, we observe that the semi-discrete flow in case of the $\P_2$ constraint is no longer constant as in Example \ref{experiment:circle} and converges with quartic rate. The probable cause for this differing behaviour lies in the different boundary conditions used, i.e. the fact that $z_0$ is not stationary for the bending energy under the semi-clamped boundary conditions from Example \ref{experiment:circle}.
    Also quartic convergence is what we also get from $\I_{h,3} z_t$, thus we have quasi-optimal convergence of $e_{ht}$. When it comes to the weaker norms shown in Table \ref{table:weak norm error helix}, in case of the $\P_1$ constraint we observe quadratic convergence for both, the $L^\infty L^2$ and $L^\infty H^1$ error. In case of the $\P_2$ constraint we observe quartic convergence for both, $\Vert \tilde e_h \Vert_{L^\infty L^2}$ and $\vert \tilde e_h \vert_{L^\infty H^1}$. Since we have
      \begin{gather*}
        \Vert e_h^n \Vert \leq \Vert \tilde e_h^n \Vert + \Vert z(t_n) - \I_{h,3} z(t_n) \Vert \leq c h^4, \\
        \vert e_h^n \vert_{H^1(I)^d} \leq \vert \tilde e_h^n \vert_{H^1(I)^d}
        + \vert z(t_n) - \I_{h,3} z(t_n) \vert_{H^1(I)^d}
        \leq c h^4 + c h^3 \leq c h^3,
      \end{gather*}
      we obtain quartic convergence in $L^\infty L^2$ and cubic convergence $L^\infty H^1$.
    \begin{filecontents*}{tables/helix_h2_combined.csv}
1.64845,1.070e+00,-,1.070e+00,-,1.070e+00,-,2.081e-01,-,2.081e-01,-,2.081e-01,-
0.82423,5.498e-01,0.96008,5.498e-01,0.96008,5.498e-01,0.96008,5.320e-02,1.96792,5.320e-02,1.96792,5.320e-02,1.96792
0.41211,2.768e-01,0.99030,2.768e-01,0.99031,2.768e-01,0.99031,1.338e-02,1.99194,1.338e-02,1.99194,1.338e-02,1.99194
0.20606,1.386e-01,0.99759,1.386e-01,0.99759,1.386e-01,0.99759,3.348e-03,1.99798,3.348e-03,1.99798,3.348e-03,1.99798
0.10303,6.934e-02,0.99940,6.934e-02,0.99940,6.934e-02,0.99940,8.374e-04,1.99943,8.375e-04,1.99941,8.375e-04,1.99942
\end{filecontents*}
\DTLloaddb[noheader,keys={h,e11,r11,e12,r12,e13,r13,e21,r21,e22,r22,e23,r23}]
    {helix_h2_DB}
    {tables/helix_h2_combined.csv}  
    \begin{center}
      \begin{minipage}{.85\linewidth}
        \resizebox{\linewidth}{!}{
          \begin{tabular}{c|cc|cc|cc|cc}\toprule
            \multirow{3}{*}{$h$}
            & \multicolumn{4}{c|}{$\mathcal{P}_1$ constraint}
            & \multicolumn{4}{c}{$\mathcal{P}_2$ constraint} \\
            \cmidrule{2-9}
            & \multicolumn{2}{c|}{$\tau = 1/10$} 
            & \multicolumn{2}{c|}{$\tau = 1/20$}
            & \multicolumn{2}{c|}{$\tau = 1/10$}
            & \multicolumn{2}{c}{$\tau = 1/20$}\\
            & $\vert e_{h} \vert_{L^\infty H^2}$ & eoc 
            & $\vert e_{h} \vert_{L^\infty H^2}$ & eoc
            & $\vert e_{h} \vert_{L^\infty H^2}$ & eoc
            & $\vert e_{h} \vert_{L^\infty H^2}$ & eoc
            \DTLforeach*{helix_h2_DB}
            {\h=h,\a=e11,\b=r11,\c=e12,\d=r12,\e=e21,\f=r21,\g=e22,\i=r22}{%
              \DTLiffirstrow{\\\cmidrule{1-9}}{\\}%
              \h & \a & \b & \c & \d & \e & \f & \g & \i
            }%
            \\\bottomrule
        \end{tabular}}
        \captionof{table}
        {
          \label{table:maxH2 error helix}
          Approximation error in Example \ref{experiment:helix} in the $L^\infty H^2$ semi-norm. The observed
          convergence rate is linear in case of the $\P_1$ constraint and
          quadratic in case of the $\P_2$ constraint.
        }
    \end{minipage}\end{center}
    \begin{filecontents*}{tables/helix_time_combined.csv}
1.64845,7.395e-01,-,7.398e-01,-,1.057e-02,-,1.077e-02,-
0.82423,3.853e-01,0.94061,3.859e-01,0.93871,8.948e-04,3.56185,1.051e-03,3.35723
0.41211,1.928e-01,0.99892,1.941e-01,0.99175,4.434e-05,4.33506,5.389e-05,4.28543
0.20606,9.303e-02,1.05131,9.538e-02,1.02494,2.371e-06,4.22516,2.711e-06,4.31314
\end{filecontents*}
\DTLloaddb[noheader,keys={h,e11,r11,e12,r12,e21,r21,e22,r22}]
    {helix_time_DB}
    {tables/helix_time_combined.csv}  
    \begin{center}\begin{minipage}{.85\linewidth}
        \resizebox{\linewidth}{!}{
          \begin{tabular}{c|cc|cc|cc|cc}\toprule
            \multirow{3}{*}{$h$}
            & \multicolumn{4}{c|}{$\mathcal{P}_1$ constraint}
            & \multicolumn{4}{c}{$\mathcal{P}_2$ constraint} \\
            \cmidrule{2-9}
            & \multicolumn{2}{c|}{$\tau = 1/1000$} 
            & \multicolumn{2}{c|}{$\tau = 1/2000$}
            & \multicolumn{2}{c|}{$\tau = 1/1000$}
            & \multicolumn{2}{c}{$\tau = 1/2000$}\\
            & $\vert e_{h} \vert_{H^1 L^2}$ & eoc 
            & $\vert e_{h} \vert_{H^1 L^2}$ & eoc
            & $\vert e_{h} \vert_{H^1 L^2}$ & eoc
            & $\vert e_{h} \vert_{H^1 L^2}$ & eoc
            \DTLforeach*{helix_time_DB}
            {\h=h,\a=e11,\b=r11,\c=e12,\d=r12,\e=e21,\f=r21,\g=e22,\i=r22}{%
              \DTLiffirstrow{\\\cmidrule{1-9}}{\\}%
              \h & \a & \b & \c & \d & \e & \f & \g & \i
            }%
            \\\bottomrule
        \end{tabular}}
        \captionof{table}
        {
          \label{table:H1L2 error helix}
          Calculated approximation error in Example \ref{experiment:helix} in the $H^1L^2$ semi-norm. For the
          $\P_1$ constraint we observe linear convergence, which is the same
          convergence rate as in Example \ref{experiment:circle}.
          For the $\P_2$ constraint however the $H^1L^2$ error is no longer zero,
          but of order $O(h^4)$ instead.
        }      
      \end{minipage}
    \end{center}

    \begin{filecontents*}{tables/helix_weak_norms.csv}
1.64845,8.004e-01,-,5.648e-01,-,9.335e-03,-,8.177e-03,-
0.82423,2.035e-01,1.97564,1.495e-01,1.91788,5.620e-04,4.05398,5.095e-04,4.00435
0.41211,5.116e-02,1.99207,3.789e-02,1.97987,3.497e-05,4.00660,3.184e-05,4.00047
0.20606,1.281e-02,1.99796,9.505e-03,1.99499,2.183e-06,4.00131,1.990e-06,4.00011
0.10303,3.203e-03,1.99949,2.378e-03,1.99875,1.364e-07,4.00031,1.243e-07,4.00003
\end{filecontents*}
\DTLloaddb[noheader,keys={h,e11,r11,e12,r12,e21,r21,e22,r22}]
    {helix_weak_norms_DB}
    {tables/helix_weak_norms.csv}  
    \begin{center}\begin{minipage}{.85\linewidth}
        \resizebox{\linewidth}{!}{
          \begin{tabular}{c|cc|cc|cc|cc}\toprule
            \multirow{2}{*}{$h$}
            & \multicolumn{4}{c|}{$\mathcal{P}_1$ constraint}
            & \multicolumn{4}{c}{$\mathcal{P}_2$ constraint} \\
            \cmidrule{2-9}
            & $\Vert \widetilde e_{h} \Vert_{L^\infty L^2}$ & eoc 
            & $\vert \widetilde e_{h} \vert_{L^\infty H^1}$ & eoc
            & $\Vert \widetilde e_{h} \Vert_{L^\infty L^2}$ & eoc
            & $\vert \widetilde e_{h} \vert_{L^\infty H^1}$ & eoc
            \DTLforeach*{helix_weak_norms_DB}
            {\h=h,\a=e11,\b=r11,\c=e12,\d=r12,\e=e21,\f=r21,\g=e22,\i=r22}{%
              \DTLiffirstrow{\\\cmidrule{1-9}}{\\}%
              \h & \a & \b & \c & \d & \e & \f & \g & \i
            }%
            \\\bottomrule
        \end{tabular}}
        \captionof{table}
        {
          \label{table:weak norm error helix}
          Calculated $L^\infty L^2$ and $L^\infty H^1$ approximation error in Example \ref{experiment:helix} for time step size $\tau = 1/20$. For the $\P_1$ constraint we observe quadratic convergence while for the $\P_2$ constraint we observe quartic convergence.
        }      
      \end{minipage}
    \end{center}
  \end{example}
  
  \begin{example}[Forced helix]
    \label{experiment:circle to helix}
    For this last experiment we want to consider a non-stationary flow. However, in order to be able to calculate the approximation error, we still need to know the continuous solution. For this, we first choose a suitable function $\widetilde z$, that is non-stationary and then construct a right-hand side $f$ of the bending problem such that $\widetilde z$ is the continuous solution. 
    We therefore set $T = 1$, $I = [0,2\pi]$ and
    \begin{align*}
      r(t) := \sqrt{1 - \frac{t^2}{4 \pi^2}}\ , \leer
      \widetilde z(x,t) 
      := \left(r(t) \cos(x), r(t) \sin(x), \frac{tx}{2\pi} \right).
    \end{align*}
    Therefore $\widetilde z$ is a function, that starts as a circle in $\real^3$ and then transforms into a helix as visualized in Figure~\ref{figure:circle to helix}. Further by definition $\widetilde z$ satisfies the arc-length constraint $\vert \widetilde z_x \vert^2 = 1$.
    \begin{figure}
      \begin{center} 
        \includegraphics[trim={5cm 15mm 5cm 30mm},clip,width=.3\textwidth]
        {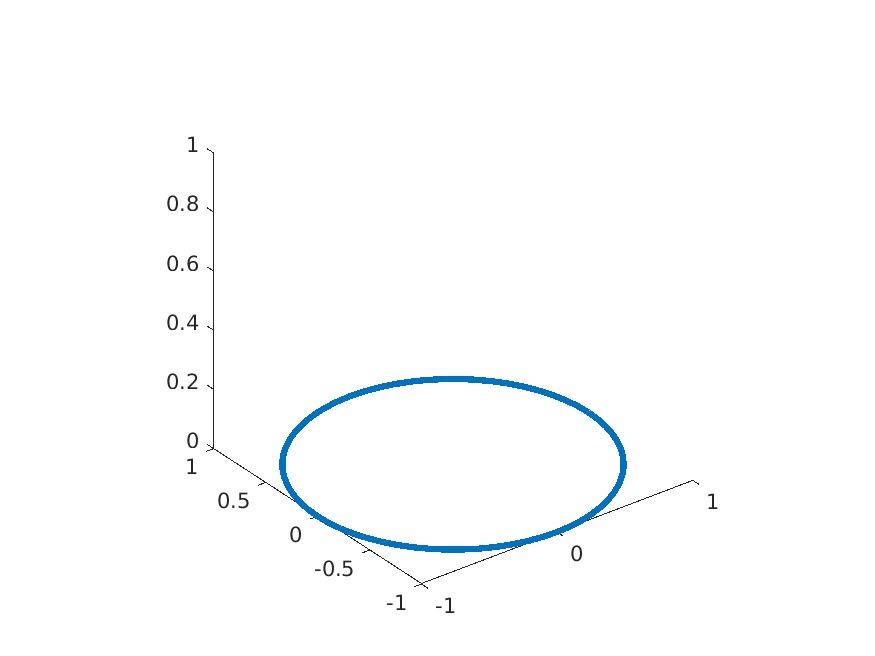}
        \includegraphics[trim={5cm 15mm 5cm 30mm},clip,width=.3\textwidth]
        {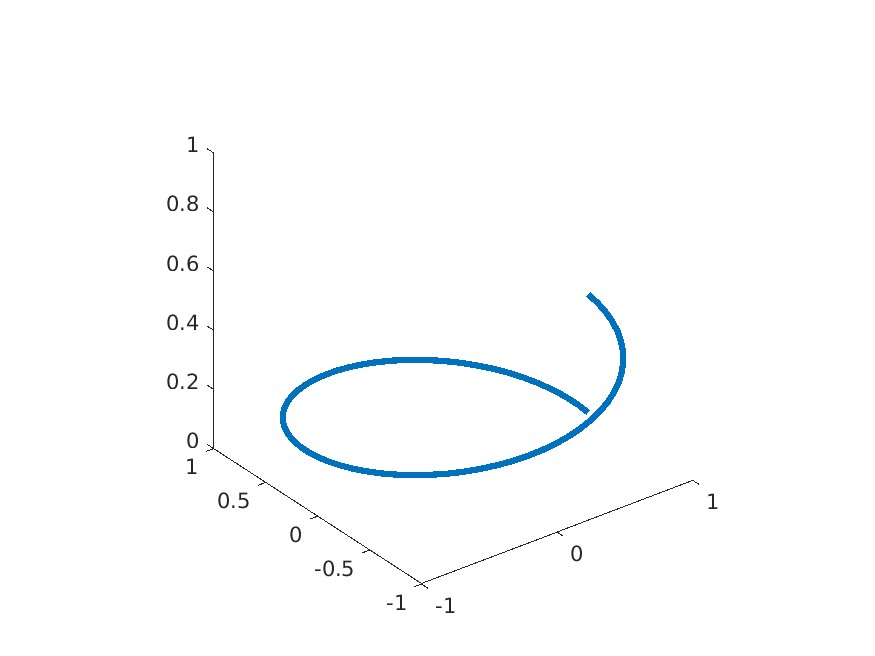}
        \includegraphics[trim={5cm 15mm 5cm 30mm},clip,width=.3\textwidth]
        {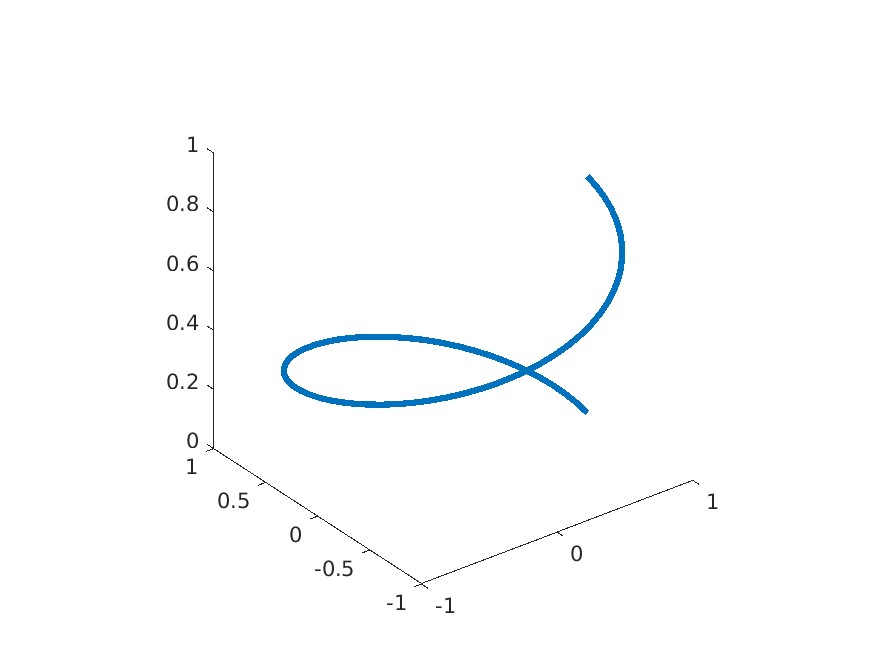}
      \end{center}
      \caption{
        Constructed continuous solution $\widetilde z(\cdot,t)$ from Example \ref{experiment:circle to helix} for $t = 0,\ 0.4,\ 0.8$.
      }  
      \label{figure:circle to helix} 
    \end{figure} \\
    We now define an operator $L: L^\infty([0,T],H^2(I)^d) \cap 
    H^1([0,T],L^2(I)^d)
    \to L^\infty([0,T],(H^2(I)^d))'$ via
    \begin{equation*}
      Lz(y) := \int_I z_t \cdot y + z_{xx} \cdot y_{xx} + \lambda(z) z_x \cdot y_x \ dx
    \end{equation*}
    with $\lambda(z) := -\int_x^b z_t \ d\sigma \cdot z_x - \vert z_{xx} \vert^2$.
    If $\lambda$ were the Lagrange multiplier corresponding to the  constraint $\vert z_x \vert^2 = 1$ for clamped boundary conditions, the original problem $(\ref{equation:elastic_flow_pde})$ could be written as $Lz = 0$. But since the Lagrange multiplier for this case is unknown, we use the one from the semi-clamped case as an approximation instead.
    Now we want to solve the equation 
    \begin{align*}
      \int_I z_t \cdot v + z_{xx} \cdot v_{xx} \ dx = f
    \end{align*}
    for all $v \in \G(z)$ with $ f = L \widetilde z$, boundary conditions $z = \widetilde z$, $z' = \widetilde z'$ on $\del I$, initial value $z(x,0) = \widetilde z(x,0)$ and $\widetilde z$ as defined above, i.e we solve
    \begin{align*}
      \int_I z_t \cdot y + z_{xx} \cdot y_{xx} \ dx
      &= \int_I \widetilde z_t \cdot y + \widetilde z_{xx} \cdot y_{xx} + \lambda(\widetilde z) \widetilde z_x \cdot y_x \ dx \\
      &= \int_I \widetilde z_t \cdot y + \widetilde z_{xx} \cdot y_{xx} - (\lambda(\widetilde z)\widetilde z_x)_x \cdot y \ dx
    \end{align*}
    for all $y \in \G(z)$ with $z$ satisfying the constraint $z_{tx} \cdot z_x = 0$ and the required boundary conditions.
    It is easy to see that $\widetilde z$ is indeed a solution to this problem.
    We discretize this problem by inserting the right-hand side $f$ into the discrete scheme $(\ref{equation:discrete_scheme})$, which we also adjust for the time dependent boundary conditions. We obtain
    \begin{align*}
      (d_t Z^{n+1}, Y) + \tau (d_tZ^{n+1}_{xx}, Y_{xx})
      = &-(Z^n_{xx}, Y_{xx}) + (F^{n+1},Y)
    \end{align*}
    with $(F^{n+1},Y) = (\widetilde z_t(t_{n+1}),Y) + (\widetilde z_{xx}(t_{n+1}), Y_{xx}) + ((\lambda(\widetilde z) \widetilde z_x)_x, Y)$.
    To simplify the right-hand side in the time stepping scheme, we first note that according to Lemma \ref{lemma:rhs_interpolation} we have $(\widetilde z_{xx},Y_{xx}) = ((\I_{h,3} \widetilde z)_{xx}, Y_{xx})$. We then approximate $\widetilde z_t$ and $(\lambda(\widetilde z) \widetilde z_x)_x$ with their respective $\P_3$-interpolants and with $U^n := \I_{h,3} \widetilde z(t_n)$, $V^n := \I_{h,3} \widetilde z_t(t_n)$ and $W^n := \I_{h,3}((\lambda(\widetilde z)\widetilde z_x)_x)(t_n)$ we obtain the modified discrete scheme
    \begin{align*}
      (d_t Z^{n+1}, Y) + \tau (d_t Z^{n+1}_{xx}, Y_{xx})
      = &-(Z^n_{xx}, Y_{xx}) + (V^{n+1}, Y) + (U^{n+1}_{xx},Y_{xx}) - (W^{n+1},Y)
    \end{align*}
    for all $Y \in \G_h(Z^n)$. With the mass matrix $M$ and second order stiffness
    matrix $S$ this can be written as:
    \begin{align*}
      Y^T (M + \tau S) d_t Z^{n+1} = Y^T(M(V^{n+1} - W^{n+1}) + S(U^{n+1} - Z^n))
    \end{align*}
    for all $Y \in \G_h(Z^n)$.
    We now set $Z^0 = \J_{h,3} \widetilde z(\cdot, 0)$ and then in every time step have to solve
    \begin{equation*}
      \begin{bmatrix}
        M + \tau S & B_n^T \\ B_n & 0
      \end{bmatrix}
      \begin{bmatrix}
        d_t Z^{n+1} \\ \Lambda^{n+1}
      \end{bmatrix}
      = \begin{bmatrix}
        M(V^{n+1} - W^{n+1})
        + S(U^{n+1} - Z^n) \\ Q^{n+1}
      \end{bmatrix},
    \end{equation*}
    where the matrix $B_n$ and the vector $Q^{n+1}$ are used to enforce the linearized constraint on the inner nodes and the boundary conditions $Z^{n+1} = \widetilde z(t_{n+1})$, $Z^{n+1}_x = \widetilde z_x(t_{n+1})$ on $\del I$. We again set $e_h^n := z(t_n) - Z^n$ and calculate the approximation errors $\vert e_h \vert_{L^\infty H^2}$ and $\vert e_h \vert_{H^1 L^2}$ as previously. The corresponding results are shown in Table \ref{table:maxH2 error circle to helix} and Table \ref{table:H1L2 error circle to helix}. The observed convergence rates for both constraints are the same as in the two stationary cases. For the approximation errors in the weaker norms, displayed in Table \ref{table:weak norm error forced helix}, we observe similar results to the stationary cases as well with quadratic convergence for the $\P_1$ constraint and quartic convergence for the $\P_2$ constraint.
    \begin{filecontents*}{tables/helix_flow_h2_combined.csv}
1.57080,1.166e+00,-,1.166e+00,-,2.228e-01,-,2.228e-01,-
0.78540,7.035e-01,0.72856,7.035e-01,0.72856,5.714e-02,1.96322,5.714e-02,1.96322
0.39270,3.631e-01,0.95403,3.631e-01,0.95403,1.438e-02,1.99081,1.438e-02,1.99081
0.19635,1.830e-01,0.98899,1.830e-01,0.98899,3.600e-03,1.99770,3.600e-03,1.99770
\end{filecontents*}
\DTLloaddb[noheader,keys={h,e11,r11,e12,r12,e21,r21,e22,r22}]
    {helix_flow_h2_DB}
    {tables/helix_flow_h2_combined.csv}
    \ \\  
    \begin{center}\begin{minipage}{.85\linewidth}
        \resizebox{\linewidth}{!}{
          \begin{tabular}{c|cc|cc|cc|cc}\toprule
            \multirow{3}{*}{$h$}
            & \multicolumn{4}{c|}{$\mathcal{P}_1$ constraint}
            & \multicolumn{4}{c}{$\mathcal{P}_2$ constraint} \\
            \cmidrule{2-9}
            & \multicolumn{2}{c|}{$\tau =$ 2e-05} 
            & \multicolumn{2}{c|}{$\tau =$ 1e-05}
            & \multicolumn{2}{c|}{$\tau =$ 2e-05}
            & \multicolumn{2}{c}{$\tau =$ 1e-05}\\
            & $\vert e_{h} \vert_{L^\infty H^2}$ & eoc 
            & $\vert e_{h} \vert_{L^\infty H^2}$ & eoc
            & $\vert e_{h} \vert_{L^\infty H^2}$ & eoc
            & $\vert e_{h} \vert_{L^\infty H^2}$ & eoc
            \DTLforeach*{helix_flow_h2_DB}
            {\h=h,\a=e11,\b=r11,\c=e12,\d=r12,\e=e21,\f=r21,\g=e22,\i=r22}{%
              \DTLiffirstrow{\\\cmidrule{1-9}}{\\}%
              \h & \a & \b & \c & \d & \e & \f & \g & \i
            }%
            \\\bottomrule
        \end{tabular}}
        \captionof{table}
        {
          \label{table:maxH2 error circle to helix}
          Approximation error for the approximation of the forced helix in Example \ref{experiment:circle to helix}
          in $L^\infty H^2$. Just like in the stationary case, we observe an
          improvement in convergence rate from linear to quadratic as we move
          from the $\P_1$ constraint to the $\P_2$ constraint.
        }    
    \end{minipage}\end{center}
    \begin{filecontents*}{tables/helix_flow_time_combined.csv}
1.57080,8.152e-01,-,8.152e-01,-,5.612e-03,-,5.612e-03,-
0.78540,4.090e-01,0.99517,4.090e-01,0.99512,3.877e-04,3.85525,3.879e-04,3.85454
0.39270,2.020e-01,1.01740,2.021e-01,1.01720,2.423e-05,4.00034,2.439e-05,3.99157
0.19635,1.005e-01,1.00780,1.005e-01,1.00698,1.991e-06,3.60546,1.570e-06,3.95718
\end{filecontents*}
\DTLloaddb[noheader,keys={h,e11,r11,e12,r12,e21,r21,e22,r22}]
    {helix_flow_time_DB}
    {tables/helix_flow_time_combined.csv}
    \begin{center}\begin{minipage}{.85\linewidth}
        \resizebox{\linewidth}{!}{
          \begin{tabular}{c|cc|cc|cc|cc}\toprule
            \multirow{3}{*}{$h$}
            & \multicolumn{4}{c|}{$\mathcal{P}_1$ constraint}
            & \multicolumn{4}{c}{$\mathcal{P}_2$ constraint} \\
            \cmidrule{2-9}
            & \multicolumn{2}{c|}{$\tau =$ 2e-05} 
            & \multicolumn{2}{c|}{$\tau =$ 1e-05}
            & \multicolumn{2}{c|}{$\tau =$ 2e-05}
            & \multicolumn{2}{c}{$\tau =$ 1e-05}\\
            & $\vert e_{h} \vert_{H^1 L^2}$ & eoc 
            & $\vert e_{h} \vert_{H^1 L^2}$ & eoc
            & $\vert e_{h} \vert_{H^1 L^2}$ & eoc
            & $\vert e_{h} \vert_{H^1 L^2}$ & eoc
            \DTLforeach*{helix_flow_time_DB}
            {\h=h,\a=e11,\b=r11,\c=e12,\d=r12,\e=e21,\f=r21,\g=e22,\i=r22}{%
              \DTLiffirstrow{\\\cmidrule{1-9}}{\\}%
              \h & \a & \b & \c & \d & \e & \f & \g & \i
            }%
            \\\bottomrule
        \end{tabular}}
        \captionof{table}
        {
          \label{table:H1L2 error circle to helix}
          $H^1L^2$ approximation error for the forced helix in Example~\ref{experiment:circle to helix}.
          The observed experimental convergence rate is linear in case of the
          $\P^1$ constraint while it is of order 4 in case of the $\P_2$
          constraint, just as in the case of the stationary helix.
        }
    \end{minipage}\end{center}
  
    \begin{filecontents*}{tables/forced_helix_weak_norms.csv}
1.57080,6.663e-01,-,5.740e-01,-,8.087e-03,-,7.817e-03,-
0.78540,2.077e-01,1.68209,1.787e-01,1.68353,5.027e-04,4.00774,5.137e-04,3.92763
0.39270,5.528e-02,1.90938,4.610e-02,1.95446,3.544e-05,3.82632,3.429e-05,3.90491
0.19635,1.404e-02,1.97705,1.161e-02,1.98899,7.476e-06,2.24506,5.121e-06,2.74337
\end{filecontents*}
\DTLloaddb[noheader,keys={h,e11,r11,e12,r12,e21,r21,e22,r22}]
    {forced_helix_weak_norms_DB}
    {tables/forced_helix_weak_norms.csv}
    \begin{center}\begin{minipage}{.85\linewidth}
        \resizebox{\linewidth}{!}{
          \begin{tabular}{c|cc|cc|cc|cc}\toprule
            \multirow{2}{*}{$h$}
            & \multicolumn{4}{c|}{$\mathcal{P}_1$ constraint}
            & \multicolumn{4}{c}{$\mathcal{P}_2$ constraint} \\
            \cmidrule{2-9}
            & $\Vert \widetilde e_{h} \Vert_{L^\infty L^2}$ & eoc 
            & $\vert \widetilde e_{h} \vert_{L^\infty H^1}$ & eoc
            & $\Vert \widetilde e_{h} \Vert_{L^\infty L^2}$ & eoc
            & $\vert \widetilde e_{h} \vert_{L^\infty H^1}$ & eoc
            \DTLforeach*{forced_helix_weak_norms_DB}
            {\h=h,\a=e11,\b=r11,\c=e12,\d=r12,\e=e21,\f=r21,\g=e22,\i=r22}{%
              \DTLiffirstrow{\\\cmidrule{1-9}}{\\}%
              \h & \a & \b & \c & \d & \e & \f & \g & \i
            }%
            \\\bottomrule
        \end{tabular}}
        \captionof{table}
        {
          \label{table:weak norm error forced helix}
          Calculated $L^\infty L^2$ and $L^\infty H^1$ approximation errors in Example~\ref{experiment:circle to helix} for time step size $\tau = $ 1e-5. The observed experimental convergence rate is quadratic in case of the $\P^1$ constraint while it is quartic in case of the $\P_2$
          constraint.
        }
    \end{minipage}\end{center}
  \end{example}
  
  % Appendix
  \renewcommand{\thesection}{A}
  \section{Appendix}

  \begin{lemma}[Interpolation stability]
    \label{lemma:interpolation stability}
    Let $I = (0,1) \subset \real$ and $\T_h = \{I\}$. Then the nodal interpolants $\I_{h,k}: C^l(I) \to \S^{k,l}(\T_h) = \P_k$ satisfy
    \begin{equation}
      \label{equation:interpolation stabilty}
      \Vert \I_{h,k} u \Vert_{W^{m,p}(I)} \leq C \Vert u \Vert_{C^l(\overline I)}
    \end{equation}
    for all $m \geq 0$, $1 \leq p \leq \infty$.
  \end{lemma}
  \begin{proof}
    This is a special case of \cite[Lemma 4.4.1]{BS08}.
  \end{proof}
  
  \begin{lemma}[Interpolation estimate]\label{lemma:interpolation_estimate}
    Let $I = \bigcup_{i = 1}^M [x_{i-1}, x_i]$ a decomposition of an interval $I$ with
    $\vert x_i - x_{i-1} \vert \leq h$ for all $i = 1,...,M$ and $u \in W^{m+1,p}(I)$ arbitrary. 
    Further, let $\I_{h,m}$ be the
    Lagrange-interpolation operator of polynomial degree $m \in \{1,2,3\}$. 
    Then $\I_{h,m}$ satisfies
    \begin{equation}
      \label{equation:interpolation estimate}
      \left( \sum_{i = 1}^M \vert u - \I_{h,m} u \vert_{W^{k,p}(I_i)}^p \right)^\frac{1}{p}
      \leq c h^{r-k} \vert u \vert_{W^{r,p}(I)}
    \end{equation}
    for all $k \in \{0,...,r\}$,
    where $r \in \{\max(1,m-1),...,m+1\}$ arbitrary.
    Further, for $k \geq 1$, the interpolant $\J_{h,3}$ satisfies the same estimate.
    For $k = 0$, from the interpolation estimate of $\I_{h,2}$, we get
    \begin{equation}
      \label{equation:interpolation estimate I}
      \Vert u - \J_{h,3} u \Vert_{L^\infty(I)}
      \leq c h^3 \vert u \vert_{H^4(I)}.
    \end{equation}
    With more regularity of $u$, the Simpson rule from Lemma \ref{lemma:new simpson rule}
    implies
    \begin{equation}
      \label{equation:interpolation estimate II}
      \Vert u - \J_{h,3} u \Vert_{L^\infty(I)}
      \leq c h^4 \vert u \vert_{W^{5,\infty}(I)}.
    \end{equation}
  \end{lemma}
  \begin{proof}
    The first estimate follows from using local estimates on each subinterval and summing
    up over all intervals. The local estimate used is a special case of
    \cite[Theorem 4.4.4]{BS08} that is obtained by using $\P_{r-1} \subset \P_m$. For $k \geq 1$, (\ref{equation:interpolation estimate}) implies for $r \in \{1,...,3\}$
    \begin{align*}
      \left( \sum_{i=1}^M \vert u - \J_{h,3}u \vert_{W^{k,p}(I_i)}^p \right)^\frac{1}{p}
      = \left( \sum_{i=1}^M \vert u' - \I_{h,2} u' \vert_{W^{k-1,p}(I_i)}^p \right)^\frac{1}{p}
      \leq c h^{r-k+1} \vert u \vert_{W^{r+1,p}(I)},
    \end{align*}
    which is exactly (\ref{equation:interpolation estimate I}). For $k = 0$ we have
    \begin{align*}
      \Vert u - \J_{h,3} u \Vert_{L^\infty(I)}
      = \left\Vert \int_a^x u' - \I_{h,2} u' \ d\sigma \right\Vert_{L^\infty(I)}
      \leq \Vert u' - \I_{h,2} u' \Vert_{L^1(I)}
      \leq c h^3 \vert u \vert_{H^4(I)}.
    \end{align*}
    Further for each $x_j$ Lemma \ref{lemma:new simpson rule} implies
    \begin{align*}
      \vert (u - \J_{h,3}u)(x_j) \vert
      = \left\vert \int_a^{x_j} u' - \I_{h,2} u' \ dx \right\vert
      \leq C h^4 \Vert D_h^4 u' \Vert_{L^\infty(I)}.
    \end{align*}
    For $x \in (x_{j}, x_{j+1})$ we have
    \begin{align*}
      \vert (u - \J_{h,3}u)(x) \vert
      &\leq \vert (u - \J_{h,3})(x_j) \vert 
      + ch \Vert (u - \J_{h,3}u)' \Vert_{L^\infty(I)} \\
      &\leq c h^4 \Vert D_h^4 u' \Vert_{L^\infty(I)}
      + ch \Vert u' - \I_{h,2} u' \Vert_{L^\infty(I)}
      \leq c h^4 \Vert u \Vert_{W^{5,\infty}(I)},
    \end{align*}
    which proves $(\ref{equation:interpolation estimate II})$.
  \end{proof}
  
  \begin{lemma}\label{lemma:L1_estimate}
    Let $u \in C^0(I), v \in C^0(I)^d$ with $\vert v(z) \vert = 1$ for all
    $z \in \N_2(\T_h)$.
    Then we have
    $$
    \Vert \I_{h,2}(uv) \Vert_{L^1(I)^d} \leq c \Vert \I_{h,2}(u) \Vert_{L^1(I)}.
    $$
  \end{lemma}
  \begin{proof}
    The statement follows from elementwise transformation onto a reference interval and using norm basic norm equivalences between finite dimensional spaces.
  \end{proof}
  
  \begin{lemma}\label{lemma:L1_estimate_2}
    There exists a constant $c > 0$ such that for all $v_h \in \S^{2,0}(\T_h)^d$
    $$
    \Vert \I_{h,2}(\vert v_h \vert^2) \Vert_{L^1(I)} \leq c \Vert v_h \Vert^2.
    $$
  \end{lemma}
  \begin{proof}
    The statement follows from elementwise transformation onto a reference interval and application of the stability estimate (\ref{equation:interpolation stabilty}) and the inverse estimate (\ref{equation:inverse estimate}).
  \end{proof}
  
  \begin{lemma}
    [Improved Simpson rule]
    \label{lemma:new simpson rule}
    Let $\T_h = \{I_i \ \vert \ i = 1,...,M\}$ denote a dissection of $I$. Let further $f \in W^{1,1}(I)$ and $g \in C^0(\overline I)$ be elementwise in $C^4$. Then
    \begin{align*}
      \left\vert \int_I f (g - \I_{h,2} g) \ dx \right\vert
      \leq c h^4 (\Vert f \Vert_{L^1(I)} \Vert D_h^4 \Vert_{L^\infty(I)}
      + \Vert f' \Vert_{L^1(I)} \Vert D_h^3 g \Vert_{L^\infty(I)}).
    \end{align*}
  \end{lemma}
  
  \begin{proof}
    Let us abbreviate
    \begin{align*}
      a_i := \fint_{I_i} f \ dx = \frac{1}{h_i} \int_{I_i} f \ dx.
    \end{align*}
    Then we have
    \begin{align*}
      \int_I f (g - \I_{h,2} g) \ dx
      = \sum_{i=1}^M \int_{I_i} (f - a_i)(g - \I_{h,2}g) \ dx
      + \sum_{i = 1}^M a_i \int_{I_i} g - \I_{h,2} g \ dx.
    \end{align*}
    The error formula for Simpson's rule, see \cite[Section 3.1]{SB02} yields
    \begin{align*}
      \left\vert \int_{I_i} g - \I_{h,2} g \ dx \right\vert
      \leq \frac{h^5}{90} \max_{x \in I_i} \vert f^{(4)}(x) \vert
      = c h^5 \Vert f^{(4)} \Vert_{L^\infty(I_i)}
    \end{align*}
    while it is well known from a Poincaré inequality that
    \begin{align*}
      \left\vert \int_{I_i} f - a_i \ dx \right\vert \leq c h_i \Vert f' \Vert_{L^1(I_i)}.
    \end{align*}
    Thus
    \begin{align*}
      \left\vert \int_I f (g - \I_{h,2} g) \ dx \right\vert
      &\leq c h \Vert f' \Vert_{L^1(I)} \Vert g - \I_{h,2} g \Vert_{L^\infty(I)}
      + c \Vert f \Vert_{L^1(I)} \sum_{i=1}^M h_i^5 \Vert D_h^4 g \Vert_{L^\infty(I_i)} \\
      &\leq c h^4 (\Vert f' \Vert_{L^1(I)} \Vert D_h^3 g \Vert_{L^\infty(I)}
      + \Vert f \Vert_{L^1(I)} \Vert D_h^4 g \Vert_{L^\infty(I)})
    \end{align*}
    where we also used Lemma \ref{lemma:interpolation_estimate}.
  \end{proof}
  
  \begin{lemma}\label{lemma:rhs_interpolation}
    Let $f \in H^2(I)^d$. Then we have
    \begin{equation*}
      \int_I v_{hxx} \cdot f_{xx} \ dx = \int_I v_{hxx} \cdot 
      (\I_{h,3} f)_{xx} \ dx
    \end{equation*}
    for all $v_h \in \S^{3,1}(\T_h)^d$.
  \end{lemma}
  
  \begin{proof}
    Let $f \in H^2(I)^d$ and $v_h \in \S^{3,1}(\T_h)^d$ arbitrary. We have
    \begin{align*}
      \int_I v_{hxx} \cdot f_{xx} \ dx
      - \int_I v_{hxx} \cdot (\I_{h,3} f)_{xx} \ dx
      = \int_I v_{hxx} \cdot (f - \I_{h,3} f)_{xx} \ dx.
    \end{align*}
    Elementwise partial integration and the fundamental theorem of calculus yield
    \begin{align*}
      \int_{I_i} v_{hxx} \cdot (f - \I_{h,3} f)_{xx} \ dx
      &= [v_{hxx} \cdot (f - \I_{h,3} f)_x ]_{x_{i-1}}^{x_i}
      - [v_{hxxx} \cdot (f - \I_{h,3} f) ]_{x_{i-1}}^{x_i} \\
      &\kurz + \int_{I_i} v_{hxxxx} \cdot (f - \I_{h,3} f) \ dx.
    \end{align*}
    Now the first summand vanishes since $(\I_{h,3} f)_x(x_i) = f'(x_i)$ for all $i$.
    Analogously the second summand vanishes since
    $(\I_{h,3} f)(x_i) = f(x_i)$ for all $i$.
    Lastly the integral term also vanishes, since $v_h\vert_{I_i} \in \P_3$ and therefore
    $(v_h \vert_{I_i})_{xxxx} \equiv 0$. Now summation over all subintervals finishes the proof.
  \end{proof}
  
  \begin{lemma}[Inverse Estimate]\label{lemma:inverse_estimates}
    Let $I = (a,b)$ be an interval and $v \in \P_m,\ m \in \mathbb{N}$.
    We then have for all $k \geq 0$ and $p,q \in [0,\infty]$ the estimate
    \begin{align}
      \label{equation:inverse estimate}
      \vert v \vert_{W^{k,p}(I)} &\leq c (b-a)^{\frac{1}{p} - \frac{1}{q} - k} \Vert v \Vert_{L^q(I)}.
    \end{align}
  \end{lemma}
  \begin{proof}
    To show this estimate, one uses an affine transformation between $I$ and the reference interval $I_0 = (0,1)$. The estimate then simply follows from the transformation theorem and basic norm equivalences in finite dimensional vector spaces. The cases $p = \infty$ and $q = \infty$ are here treated via a case distinction.
  \end{proof}  
  
  \begin{lemma}[Gagliardo-Nirenberg inequality]\label{lemma:gagliardo_nirenberg}
    Let $I = (a,b) \subset \real$ and $u \in H^2(I)$. Then we have
    \begin{equation*}
      \Vert u' \Vert_{L^2(I)}
      \leq C \vert u \vert_{H^2(I)}^\frac{1}{2} \Vert u \Vert_{L^2(I)}^\frac{1}{2}
      + C \Vert u \Vert_{L^2(I)}.
    \end{equation*}
  \end{lemma}
  \begin{proof}
    A proof can be found in \cite[Theorem 1.3]{LZ22}.
  \end{proof}
  
  \begin{lemma}
    \label{corollary:GN}
    For all $\varepsilon > 0$ there exists $c_\varepsilon > 0$ such that for all $u \in H^2(I)$ we have
    \begin{align*}
      \Vert u' \Vert_{L^\infty(I)}^2
      \leq \varepsilon \Vert u'' \Vert^2 + c_\varepsilon \Vert u \Vert^2
    \end{align*}
  \end{lemma}
  
  \begin{proof}
    The proof follows immediately from the compactness of the embedding $H^2(I) \embedscomp W^{1,\infty}(I)$ and Ehrling's lemma, see \cite[Theorem 7.30]{RR04}.
  \end{proof}
  
  \begin{acknowledgement}
    Financial support by the German Research Foundation (DFG) via
    research unit FOR 3013 
    \textit{Vector- and tensor-valued surface PDEs} (417223351)
    is gratefully acknowledged.
  \end{acknowledgement}
     
  \printbibliography
\end{document}

%% file: macros_ds.tex
\def\del{\partial}

\def\embeds{\hookrightarrow}
\def\embedscomp{\hookrightarrow \hookrightarrow}

\def\real{\mathbb R}

\def\G{\mathcal G}
\def\I{\mathcal I}
\def\J{\mathcal J}
\def\M{\mathcal M}
\def\N{\mathcal N}
\def\P{\mathcal P}
\def\S{\mathcal S}
\def\T{\mathcal T}

\def\leer{\hspace{1 cm}}
\def\kurz{\hspace{0.5 cm}}

\newcommand{\Rom}[1]{\textup{\uppercase\expandafter{\romannumeral #1\relax}}}

\theoremstyle{plain}
\newtheorem{definition}{\begin{large}Definition\end{large}}[section]
\newtheorem{theorem}[definition]{\begin{large}Theorem\end{large}}

\newtheorem{lemma}[definition]{\begin{large}Lemma\end{large}}

\newtheorem{proposition}[definition]{\begin{large}Proposition\end{large}}

\theoremstyle{definition}
\newtheorem{example}[definition]{\begin{large}Example\end{large}}
\newtheorem*{acknowledgement}{\begin{large}Acknowledgement\end{large}}